\documentclass[psamsfonts,11pt,reqno]{amsart}
\usepackage[english]{babel}
\usepackage{amsmath,amsfonts,amssymb,amsthm,microtype,graphicx,mathtools,pgfplots,enumitem}
\usepackage[mathscr]{eucal}
\usepackage{footnote}
\usepackage{sepfootnotes}
\usepackage{blindtext}
\usepackage[hyphens]{url}
\usepackage{float}

\DeclarePairedDelimiter\norm{\lVert}{\rVert}
\usepackage{lmodern}
\theoremstyle{definition}
\newtheorem{theorem}{Theorem}[section]

\newtheorem{lemma}[theorem]{Lemma}

\newtheorem{proposition}[theorem]{Proposition}
\newtheorem{remark}[theorem]{Remark}
\numberwithin{equation}{section}
\usepackage{comment}
\usepackage{hyperref}

\def\eps{{\varepsilon}}
\def\N{\mathbb{N}}

\def\R{\mathbb{R}}

\def\EE{\mathcal{E}}

\def\HH{\mathcal{H}}

\def\KK{\mathcal{K}}

\def\Ws{W_{1+s}}
\def\Wm{W_{2m}}
\def\W{\widetilde W}
\def\y{\lvert y\rvert^a}
\def\aB{_{L^2( B_1,a)}}
\def\ab{_{L^2(\partial B_1,a)}}

\def\vf{\varphi}

\newcommand{\be}{\begin{equation}}
	\newcommand{\ee}{\end{equation}}
\newcommand{\bea}{\begin{equation*}\begin{aligned}}
		\newcommand{\eea}{\end{aligned}\end{equation*}}

\begin{document}
	\title[Epiperimetric inequalities in the fractional obstacle problem]{Epiperimetric inequalities in the obstacle problem for the fractional Laplacian}
	\author{MATTEO CARDUCCI}
	\address{Scuola Normale Superiore, Piazza dei Cavalieri 7, 56126, Pisa, Italy}
	\email{\href{mailto:matteo.carducci@sns.it}{matteo.carducci@sns.it}}
	
	\begin{abstract} Using epiperimetric inequalities approach, we study the obstacle problem $\min\{(-\Delta)^su,u-\varphi\}=0,$ for the fractional Laplacian $(-\Delta)^s$ with obstacle $\vf\in C^{k,\gamma}(\R^n)$, $k\ge2$ and $\gamma\in(0,1)$. 
		
		We prove an epiperimetric inequality for the Weiss' energy $\Ws$ and a logarithmic epiperimetric inequality for the Weiss' energy $\Wm$. Moreover, we also prove two epiperimetric inequalities for negative energies $\Ws$ and $\Wm$.
		
		By these epiperimetric inequalities, we deduce a frequency gap and a characterization of the blow-ups for the frequencies $\lambda=1+s$ and $\lambda=2m$. Finally, we give an alternative proof of the regularity of the points on the free boundary with frequency $1+s$ and we describe the structure of the points on the free boundary with frequency $2m$, with $m\in\N$ and $2m\le k.$
	\end{abstract}
	\maketitle
	%\tableofcontents
	\section{Introduction}
	\subsection{Obstacle problem for the fractional Laplacian} 
	Let $\varphi\in C^{k,\gamma}(\R^n)$, $k\ge2$ and $\gamma\in(0,1)$, that decays rapidly at infinity, we consider a solution of the obstacle problem for the fractional Laplacian with obstacle $\varphi$, that is a function $u:\R^n\to\R$ such that %decays at infinity and 
	$$\min\{(-\Delta)^su,u-\varphi\}=0,$$ i.e.
	\be\begin{cases}\label{fract}
		u(x)\ge\vf(x) & \mbox{in } \R^n\\
		(-\Delta)^s u(x)=0 & \mbox{in } \{u(x)>\varphi(x)\}\\
		(-\Delta)^s u(x)\ge0 & \mbox{in } \R^n\\
		%\lim_{\lvert x \rvert \to +\infty}u(x)=0,
	\end{cases}\ee with $s\in(0,1)$. The fractional Laplacian $(-\Delta)^s$ defined as $$ (-\Delta)^su(x) :=c_{n,s}\mbox{P.V.}\int_{\R^n}\frac{u(x)-u(y)}{\lvert x-y\rvert^{n+2s}}\,dy,$$ where $c_{n,s}=2^{2s}s\frac{\Gamma(\frac{n+2s}{2})}{\Gamma(1-s)}\pi^{-\frac n2}$ is a normalization constant. 
	
	The aim of the paper is to established the optimal regularity of the solution and to describe the structure and the regularity of the free boundary $$\Gamma(u):=\partial\Lambda(u)$$ where $$\Lambda(u):=\{x\in\R^n:u(x)=\varphi(x)\}$$ is the contact set.
	
	\subsection{The extension operator $L_a$} To study this problem, we will use the Caffarelli-Silvestre extension of $u$. As in \cite{cs07}, we consider the function $\widetilde u:\R^n\times\R\to\R$ satisfying
	$$\begin{cases}
		L_a \widetilde u(X)=0& \mbox{in } \R^{n+1}_+\\
		\widetilde u(x,y)=\widetilde u(x,-y)& \mbox{in } \R^{n+1}\\
		\widetilde u(x,0)=u(x) & \mbox{in } \R^n,\\
		%\lim_{\lvert X \rvert \to +\infty}\widetilde u(X)=0, 
	\end{cases}$$ where $X=(x,y)\in\R^{n+1}_+:=\R^n\times (0,+\infty)$ and $$L_a \widetilde u(X)=\lvert y\rvert^{-a}\mbox{div}_{x,y}(\y\nabla_{x,y} \widetilde u)=\Delta_x\widetilde u+\frac{a}{y}\partial_y \widetilde u+\Delta_y \widetilde u.$$ According to \cite{cs07} (see also \cite{ros}), if we choose $a:=1-2s\in(-1,1)$, we get \bea(-\Delta)^su(x)=-d_s\lim_{y\to0^+}(\y\partial_y \widetilde u(x,y)),\eea with $d_s>0$, i.e. $(-\Delta)^s$ is a Dirichlet-to-Neumann map for $L_a$.
	Now, since $$\mbox{div}_{x,y}(\y\nabla_{x,y} \widetilde u)=2\lim_{y\to0^+}(\y\partial_y \widetilde u(x,y))\HH^n|_{\{y=0\}}$$ in distributional sense (see Proposition \ref{prop1}), we get that the problem \eqref{fract} is equivalent to \be\begin{cases}\label{fract2}
		\widetilde u(x,0)\ge\vf(x) & \mbox{in } \R^n\\
		\widetilde u(x,y)=\widetilde u(x,-y) &\mbox{in } \R^{n+1}\\
		-L_a \widetilde u(x,y)=0 & \mbox{in } \R^{n+1}\setminus \{u(x,0)=\varphi(x)\}\\
		-L_a \widetilde u(x,y)\ge0 & \mbox{in } \R^{n+1},\\
		%\lim_{\lvert X \rvert \to +\infty}\widetilde u(X)=0,
	\end{cases}\ee where $\varphi\in C^{k,\gamma}(\R^n)$, $k\ge2$ and $\gamma\in(0,1)$.
	
	In particular, when $s=\frac12$, i.e. $a=0$ and $L_a=\Delta$, the problem \eqref{fract2} is the thin obstacle problem (also know as Signorini problem).

	Localizing the problem in $B_1\subset\R^{n+1}$, the solution of \eqref{fract2} can be obtained by minimizing the functional \be\label{e}\EE (v)=\int_{B_1}\lvert \nabla v\rvert^2\y\,dX\ee among the admissible functions $$\KK_c^\vf:=\{v\in H^1(B_1,a): v\ge\varphi \mbox{ on } B'_1,\ v=c \mbox{ on } \partial B_1,\ v(x,y)=v(x,-y) \},$$ where $B_1':=B_1\cap\{y=0\}$ and $c$ is the trace on $\partial B_1$ of $u$.
	
	Here we denote by $H^1(\Omega,a):=H^1(\Omega,\y)$ the weighted Sobolev space. Similarly, $L^2(\Omega,a):=L^2(\Omega,\y)$ is the weighted Lebesgue space.
	
	In the following, with a slight abuse of notation, we denote by $u$ the $L_a-$extension in $\R^{n+1}$ of $u$, and we suppose that $u\in H^1_{loc}(\R^{n+1},a)$.
	
	Moreover, with a slight abuse of notation, we also denote by $x_0\in\R^n$ the point $(x_0,0)\in\R^{n}\times \{0\}.$
	
	\subsection{Obstacle $\vf\equiv0$} 
	%Moreover, if we consider a solution of \eqref{fract2} with obstacle $\vf\not\equiv0$, then the blow-up of $u$ is a solution of \eqref{fract4} with $\vf\equiv0$.
	
	We will say that $u$ is a solution with 0 obstacle, if solves \eqref{fract2} with $\varphi\equiv0$, i.e. if $u$ satisfies
	\be\begin{cases}\label{fract4}
		\widetilde u(x,0)\ge0 & \mbox{in } \R^n\\
		\widetilde u(x,y)=u(x,-y) &\mbox{in } \R^{n+1}\\
		-L_a \widetilde u(x,y)=0 & \mbox{in } \R^{n+1}\setminus \{u(x,0)=0\}\\
		-L_a \widetilde u(x,y)\ge0 & \mbox{in } \R^{n+1}.
	\end{cases}\ee	
	Moreover, we denote by \bea \KK_c:=\{v\in H^1(B_1,a): v\ge0 \mbox{ on } B'_1,\ v=c \mbox{ on } \partial B_1,\ v(x,y)=v(x,-y) \}\eea the set of admissible function with $\vf\equiv0$, then $u$ is a minimum of the functional \eqref{e} in the class $\mathcal{K}_c$, with $c=u|_{\partial B_1}$.
	
	\subsection{Reduction to 0 obstacle} Let $u$ be a solution of \eqref{fract2} with obstacle $\varphi\in C^{k,\gamma}(\R^n)$, $ k\ge2$ and $\gamma\in(0,1)$. Let $q_k^{x_0}(x)$ be the $k$-th Taylor polynomial of $\vf$ at $x_0\in\Gamma(u)$ and $\widetilde q_k^{x_0}(x,y)$ be a polynomial of degree $k$ and the $L_a-$harmonic extension of $q_k^{x_0}(x)$ (see Lemma \ref{ex}). Then $\widetilde q_k^{x_0}(x,y)$ solves the following problem
	$$\begin{cases}
		L_a \widetilde  q_k^{x_0}(x,y)=0& \mbox{in } \R^{n+1}\\
		\widetilde q_k^{x_0}(x,y)=\widetilde  q_k^{x_0}(x,-y)& \mbox{in } \R^{n+1}\\
		\widetilde q_k^{x_0}(x,0)=q_k^{x_0}(x) & \mbox{in } \R^n\\
		\lvert\vf(x)-q_k^{x_0}(x)\rvert\le C\lvert x-x_0\rvert^{k+\gamma} & \mbox{in } \R^n.
	\end{cases}$$ We can define $$ u^{x_0}(x,y)=u(x,y)-\widetilde q_k^{x_0}(x,y)-(\vf(x)-q_k^{x_0}(x)),$$ then $u^{x_0}$ solves the following problem \be\begin{cases}\label{fract5}
		u^{x_0}(x,0)\ge0 & \mbox{in } \R^n\\
		u^{x_0}(x,y)=u(x,-y) &\mbox{in } \R^{n+1}\\
		-L_a  u^{x_0}(x,y)=h(x,y) & \mbox{in } \R^{n+1}\setminus \{u(x,0)=0\}\\
		-L_a  u^{x_0}(x,y)\ge h(x,y) & \mbox{in } \R^{n+1},
	\end{cases}\ee where $h(x,y)=h^{x_0}(x,y):= \Delta_x(\vf(x)-q_k^{x_0}(x))$.

	Notice that starting from an obstacle problem with obstacle $\vf\not\equiv0$, we have reduced the problem to the case  $\vf\equiv0$, where the right-hand side in the third and fourth line of \eqref{fract5} is not 0. However, the function $h=h^{x_0}$ is very small near $x_0$. Precisely, by the $C^{k,\gamma}-$regularity of $\vf$,  \be\label{estimate}\lvert h(x,y)\rvert\le C\lvert x-x_0\rvert^{k+\gamma-2}.\ee 
	
	Notice that the function $u^{x_0}$ inherits the local behavior of $u$. In what follows we will study the properties of the functions $u^{x_0}$ at all the points $x_0$ on the boundary $\Gamma(u)$ of the contact set $\Lambda(u)$.

	\subsection{State of art} In this section we give a brief overview on the state of the art of the obstacle for the fractional Laplacian; for more details we refer to \cite{dasa18} for $ s\in(0,1)$, and to \cite{psu12}, \cite{survey} for the case $s=\frac12$.
	
	The obstacle problem for the fractional Laplacian was studied by Silvestre in \cite{sil07}, where it was established the almost-optimal regularity $C^{1,\alpha}$ of the solution for all $\alpha\in(0,s)$, in the case $\vf\in C^{2,1}(\R^n)$. Moreover, in the same paper, it is proved that if $\vf\in C^{1,\beta}(\R^n)$, then $u\in C^{1,\alpha}(\R^n)$ for all $\alpha\in(0,\min\{s,\beta\})$.

	Later, in \cite{css08}, Caffarelli, Salsa and Silvestre proved the optimal regularity $C^{1,s}$ of the solution for $\varphi\in C^{2,1}(\R^n)$, thus generalizing the result previously obtained by Athanasopoulos and Caffarelli in \cite{ac04} in the case $s=\frac12$ and $\vf\equiv0$. They use a modification of the following "pure" Almgren's frequency function \be\label{alm} N^{x_0}(r,u):=\frac{r\int_{ B_r(x_0)} \lvert \nabla u \rvert^2\y\,dX}{\int_{\partial B_r(x_0)} u ^2\y\,d\mathcal{H}^n},\ee which is monotone in $r$ provided that $\vf\equiv0$. Precisely, setting 
	\be \label{H}H^{x_0}(r,v)=\int_{\partial B_r(x_0)} v ^2\y\,d\mathcal{H}^n,\ee 
	they introduce the following generalized Almgren's frequency function: \bea \Phi^{x_0}(r,v)=(r+Cr^{1+p})\frac{d}{dr}\log\max\{H^{x_0}(r,v),r^{n+a+2(k+\gamma-p)}\}\eea with $k=2$, $\gamma=1$ and $p=1$. This function is monotone in $r$ also when $\varphi\not\equiv0$, with $\varphi\in C^{k,\gamma}(\R^n)$. Here we denoted by $v=u^{x_0}$, with $x_0\in\Gamma(u)$, a solution to a similar problem to \eqref{fract5}.
	
	As a consequence of the monotonicity of $\Phi^{x_0}$, they obtain that if $x_0\in \Gamma(u)$ is a free boundary point such that
	$$\Phi^{x_0}(0^+,v):=\lim_{r\to0^+}\Phi^{x_0}(r,v)=n+a+2\lambda,$$
	for some $\lambda\in(0,2)$, then the rescalings 
	$$v_{r,x_0}(x,y):=\frac{v(x_0+rx,ry)}{\frac{1}{r^{n+a}}\int_{\partial B_r(x_0)} u ^2\y\,d\mathcal{H}^n},$$ 
	converge, up to a subsequence, to a function $v_{0,x_0}$ which is a $\lambda-$homogeneous solution of \eqref{fract4} (with 0 obstacle) (see Lemma 6.1 and Lemma 6.2 in \cite{css08}).
	
	Thus, the following set of free boundary points is well-defined: 
	$$\Gamma_\lambda(u):= \{x_0\in\Gamma(u): v_{0,x_0} \mbox{ is } \lambda-\mbox{homogeneous}\}$$ for every $\lambda\in(0,2)$. Moreover, it was shown in \cite{css08} that $\Gamma_\lambda(u)=\emptyset$ for all $\lambda\in(0,2)\setminus\{1+s\}$. 
	
	We will call $\Gamma_{1+s}(u)$ the set of regular points and we will denote it by $\mbox{Reg}(u)$ since it is known to be locally a $C^{1,\alpha}$ submanifold of dimension $n-1$ (see \cite{css08}). This is a generalization of the result of Athanasopoulos, Caffarelli and Salsa obtained in \cite{acs08} in the case $s=\frac12$ and $\vf\equiv0$. \medskip
	
	In the case $\vf\equiv0$, we can consider the "pure" Almgren's frequency function as in \eqref{alm}, and the free boundary can be decomposed as \bea\Gamma_\lambda(u)&:=\{x_0\in\Gamma(u): u_{0,x_0} \mbox{ is } \lambda-\mbox{homogeneous}\}\\&=\{x_0\in\Gamma(u): N^{x_0}(0^+,u)=\lambda\}\\&=\{x_0\in\Gamma(u): \Phi^{x_0}(0^+,u)=n+a+2\lambda\}.\eea
	Therefore
	$$\Gamma(u)=\mbox{Reg}(u)\cup\mbox{Sing}(u)\cup\mbox{Other}(u),$$ where $\mbox{Reg}(u)=\Gamma_{1+s}(u)$ are the regular points, $\mbox{Sing}(u)=\bigcup_{m\in\N}\Gamma_{2m}$ are the so-called singular points, and $\mbox{Other(u)}$ are all the remaining points in $\Gamma(u)$.
	
	For the singular points, in the case $s=\frac12$, in \cite{gp09}, Garofalo and Petrosyan proved that $\mbox{Sing}(u)$ is contained in the union of at most countably many submanifolds of class $C^{1}$.
	In \cite{gr17}, using the monotonicity of the generalized Almgren's frequency for $k\ge2$, $\gamma\in(0,1)$ and $p$ small enough,  Garofalo and Ros-Oton extended the structure of singular set to any $s\in(0,1)$. Indeed in the case $\vf\in C^{k,\gamma}(\R^n)$, $k\ge2$ and $\gamma\in(0,1)$, they proved that $$\bigcup_{\{m\in\N:\ 2m\le k\}}\Gamma_{2m}(u)$$ is contained in the union of countably many submanifolds $C^{1}$, where the bound $2m\le k$ is needed in order to assure the existence of blow-ups. Thus, they improved the result previously obtained in \cite{gp09} (in the case $s=\frac12$), where more regularity of the obstacle $\vf$ was required.	
	
	Moreover, Focardi and Spadaro described the entire free boundary, up to sets of null $\HH^{n-1}$ measure, in the case $\vf\equiv 0$ in \cite{fs18} and in the case $\vf\not\equiv 0$ in \cite{fs20}.
	\medskip
	
	An alternative proof of the regularity and structure of the free boundary uses epiperimetric inequalities approach for the Weiss' energy 
	%$W_\lambda(u)=W_\lambda(1,u)$, where 
	\be\label{W}W^{x_0}_\lambda(r,u):=\frac{1}{r^{n+a+2\lambda-1}}\int_{B_r(x_0)}\lvert\nabla u\rvert^2\y\,dX-\frac{\lambda}{r^{n+a+2\lambda}}\int_{\partial B_r(x_0)}u^2\y\,d\mathcal{H}^n .\ee %In the following we drop the dependence on $x_0$ if $x_0=0$.
	
	In the case $s=\frac12$ and $\vf\equiv0$, Garofalo, Petrosyan and Smit Vega Garcia in \cite{gps16} and Focardi and Spadaro in \cite{fs16} proved an epiperimetric inequality for $\Ws$ to deduce the regularity $C^{1,\alpha}$ of the regular points $\mbox{Reg}(u)$. In the case $s\in(\frac12,1)$ and $\vf\not\equiv 0$, using epiperimetric inequalities approach, the same regularity for $\Gamma_{1+s}(u)$  was established by Garofalo, Petrosyan, Pop and Smit Vega Garcia in \cite{gpps17} (see also \cite{geraci} for $\vf\equiv 0$ and $s\in(0,1)$). The regularity of the free boundary $\Gamma_{1+s}(u)$ in the case of more general degenerate elliptic operators for variable coefficients was established recently by Banerjee, Buseghin and Garofalo in \cite{bbg22}, using again an epiperimetric inequality.
	
	Following epiperimetric inequalities approach, Colombo, Spolaor and Velichkov in \cite{csv17} give an alternative proof of the structure of singular set $\mbox{Sing}(u)$ in the case $s=\frac12$ and $\vf\equiv0$. They improved the regularity of the manifolds that contains the singular set up to $C^{1,\log}$, due to the logarithmic epiperimetric inequality for $\Wm$.
	
	\subsection{Main results.}The goal of this paper is to generalize the epiperimetric inequalities that we already know for the thin obstacle problem $s=\frac12$, to any $s\in(0,1)$. With this generalization, we can deduce the previous results of regularity and structure of the free boundary, even for non-zero obstacles $\vf\not\equiv 0$, with $\vf\in C^{k,\gamma}(\R^n)$, $k\ge2$ and $\gamma\in(0,1)$.
	In particular, we prove an epiperimetric inequality for $\Ws$ and a logarithmic epiperimetric inequality for $\Wm$, for each $s\in(0,1)$.

	Before we state our main results (Theorem \ref{thm1} and Theorem \ref{thm2} below), we recall that
	%	First we recall that 
	\be\label{g}\KK_c:=\{v\in H^1(B_1,a): v\ge0 \mbox{ on } B'_1,\ v=c \mbox{ on } \partial B_1,\ v(x,y)=v(x,-y) \}\ee is the set of admissible function, for each $c\in H^1(\partial B_1,a)$. We will also denote by $W_\lambda(u)$ the Weiss' energy $W_\lambda^{x_0}(r,u)$, when $x_0=0$ and $r=1$. 
	
	\begin{theorem}[Epiperimetric inequality for $\Ws$]\label{thm1} Let $\mathcal{K}_c$ defined in \eqref{g} and $z=r^{1+s}c(\theta)\in\mathcal{K}_c$ be the $(1+s)-$homogeneous extension in $\mathbb{R}^{n+1}$ of a function $c\in H^1(\partial B_1,a)$. Therefore there is $\zeta\in\mathcal{K}_c$ such that $$\Ws(\zeta)\le (1-\kappa)\Ws(z),$$ with $\kappa=\frac{1+a}{2n+a+5}.$ 
	\end{theorem}
	
	\begin{theorem}[Logarithmic epiperimetric inequality for $\Wm$]\label{thm2} Let $\mathcal{K}_c$ defined in \eqref{g} and $z=r^{2m}c(\theta)\in\mathcal{K}_c$ be the $2m-$homogeneous extension in $\mathbb{R}^{n+1}$ of a function $c\in H^1(\partial B_1,a)$. We also suppose that exists a constant $\Theta>0$ such that \be\label{cond}\norm c_{L^2(\partial B_1,a)}^2\le\Theta\mbox{ and }\lvert W_{2m}(z)\rvert\le\Theta,\ee then there is $\zeta\in\mathcal{K}_c$ such that$$W_{2m}(\zeta)\le W_{2m}(z)(1-\varepsilon\Theta^{-\beta}\lvert W_{2m}(z)\rvert^\beta),$$ with $\beta=\frac{n-1}{n+1}$ and $\eps=\eps(n,m,a)>0$ small enough.
	\end{theorem}
	The first inequality was originally proved in \cite{gps16} and in \cite{fs16} for $s=\frac12$ and generalized to any $s\in(0,1)$ in \cite{gpps17} and in \cite{geraci}. For the proof, we use an alternative method, that follows the idea in \cite{csv17}, decomposing the datum $c\in H^1(\partial B_1,a)$ in terms of eigenfunctions of $L_a$ restricted to $\partial B_1$. 
	
	The second inequality was originally proved in \cite{csv17} for $s=\frac12$, but this is a new result for each $s\in(0,1)$.\medskip
	
	Moreover, we will give a proof of two epiperimetric inequalities for negative energies\footnote{Only relevant if $W(z)<0$} for $\Ws$ and for $\Wm$ (see Theorem \ref{thm4} and Theorem \ref{thm5}), originally proved for $s=\frac12$ in \cite{car23} and in \cite{csv17} respectively. 
	Using this two epiperimetric inequalities for negative energies, together with the first two epiperimetric inequalities in Theorem \ref{thm1} and Theorem \ref{thm2}, we deduce a frequency gap.
	\begin{proposition}[Frequency gap]\label{gap} 
		Let $u$ be a solution of \eqref{fract2} and $v=u^{x_0}$ be the solution of \eqref{fract5} with $x_0\in\Gamma(u)$, $\vf\in C^{k,\gamma}(\R^n)$, $k\ge2$ and $\gamma\in(0,1)$.
		
		If $v$ is $\lambda-$homogeneous with $\lambda<k+\gamma$, then $$\lambda\not\in(2s,1+s)\cup(1+s,2)\cup\bigcup_{m\in\N}\left((2m-c^-_{m,a})\cup(2m+c^+_{m,a})\right),$$ for some constants $c^\pm_{m,a}>0$, that depend only on $n,m$ and $a$.
	
	In particular, if $u$ is a $\lambda-$homogeneous solution of \eqref{fract4} (with 0 obstacle) with $\lambda>0$, then the same conclusion hold for $u$.

\end{proposition}
	
	Notice that $\lambda\not\in\bigcup_{m\in\N}\left((2m-c^-_{m,a})\cup(2m+c^+_{m,a})\right)$ is a new result for any $s\in(0,1)$ and it is originally proved in \cite{csv17} for $s=\frac12$ and $\vf\equiv0$. %Actually, we already knew that $\lambda\not\in(0,1+s)\cup(1+s,2)$ (see \cite{css08}), but we propose an alternative proof of this result, that uses the epiperimetric inequalities approach. 
	Observe that the function $-\lvert y\rvert^{2s}$ is a solution of \eqref{fract4} (with 0 obstacle), then we are able to prove the best frequency gap around $1+s$. \medskip
	
	Furthermore, we use the epiperimetric inequalities to deduce a characterization of the $\lambda-$homogeneous solutions of \eqref{fract4} (with 0 obstacle), in the case $\lambda=1+s$ and $\lambda=2m$, as described in the following proposition.
	\begin{proposition}\label{cha}
		Let $u$ be a $\lambda-$homogeneous solution of \eqref{fract4} (with 0 obstacle).
		\begin{enumerate}
			\item If $\lambda=1+s$, then $u=Ch_e^s$, for some $C\ge0$ and $e\in \partial B'_1$, where 
			\be\label{he}h_e^s(x,y)=\left(s^{-1}(x\cdot e)-\sqrt{(x\cdot e)^2+y^2}\right)\left(\sqrt{(x\cdot e)^2+y^2}+(x\cdot e)\right)^s.\ee
			\item If $\lambda=2m$, then $u=p_{2m}$, for some $p_{2m}$ that is a polynomial $2m-$homogeneous and $L_a-$harmonic, with $p_{2m}\ge0$ on $B'_1$.
		\end{enumerate}
		In particular we characterized the blow-ups of a solution $u$ at $x_0\in\Gamma(u)$ with frequency $\lambda=1+s$ or $\lambda=2m.$
	\end{proposition}\medskip

	Finally we use the epiperimetric inequality for $\Ws$ in Theorem \ref{thm1} and the logarithmic epiperimetric inequality for $\Wm$ in Theorem \ref{thm2} to get the regularity and the structure of the free boundary. In particular we prove the regularity of the points on the free boundary with frequency $1+s$ , denoted by $\Gamma_{1+s}(u)$, and we describe the structure of the points with frequency $2m$, denoted by $\Gamma_{2m}(u)$, with $2m\le k.$ See \eqref{gamma} below for the definition of $\Gamma_{\lambda}(u)$.
	\begin{theorem}\label{thm3} Let $u$ be a solution of \eqref{fract2} with obstacle $\vf\in C^{k,\gamma}(\R^n)$, $k\ge2$ and $\gamma\in(0,1)$. \begin{enumerate}
			\item $\Gamma_{1+s}(u)$ is locally a $C^{1,\alpha}$ submanifold of dimension $n-1$, for some $\alpha>0$, i.e. for all $x_0\in\Gamma_{1+s}(u)$, there is $\rho>0$ and $g:U\subset\R^{n-1}\to\R$ of class $C^{1,\alpha}$ such that $$\Gamma_{1+s}(u)\cap B'_\rho(x_0)=\{(x,y)\in\R^{n+1}:x_n=g(x_1,\ldots,x_{n-1})\}\cap B'_\rho(x_0).$$
			\item $\Gamma_{2m}(u)$ is contained in the union of at most countably many submanifolds of class $C^{1,\log}$, for all $2m\le k$. In particular, for such $m$, we have that for every $j\in\{0,\ldots n-1\},$ 
			$$\Gamma_{2m}^j(u):=\{x_0\in \Gamma_{2m}:d_{2m}^{x_0}=j\}$$ is locally contained in a $C^{1,\log}$ submanifold of dimension $j$, where $d_{2m}^{x_0}$ is defined in \eqref{d} below.
			
			In particular, when $\vf\in C^{\infty}(\R^n)$, the singular sets $\cup_{m\in \N}\Gamma_{2m}^j(u)$ is contained in the union of countably many $j$-dimensional submanifolds of class $C^{1,\log}$. Moreover, the singular set $\mbox{Sing}(u)=\cup_{m\in \N}\Gamma_{2m}(u)$ is contained in the union of countably many submanifolds of class $C^{1,\log}$.
		\end{enumerate}
		
	\end{theorem}
	Here we have defined \be\label{d}d_{2m}^{x_0}:=\mbox{dim}\{\xi\in\mathbb{R}^{n}:\xi\cdot\nabla_xp_{2m}^{x_0}(x,0)\equiv 0\}\ee to be the dimension of the "tangent plane", where $p_{2m}^{x_0}$ is the unique homogeneous blow-up (see Proposition \ref{cha}) of $u$ at $x_0\in\Gamma_{2m}(u)$.
	
	In particular, we improve the regularity of submanifolds that contain $\Gamma_{2m}^j(u)$ from $C^{1}$ (proved in \cite{gr17}) to $C^{1,\log}$.
	\subsection{Structure of the paper} The paper is organized as follows.
	\begin{itemize}
		\item In Section \ref{prel}, we recall the generalized Almgren's frequency function, the Weiss' energy $W$ for 0 obstacle and the Weiss' energy $\W$ for non-zero obstacle.
		We also introduce the operator $L_a^S$, i.e. $L_a$ restricted to $\partial B_1$, its eigenfunctions and its relation to the Weiss' energy $W$. Finally, we recall the properties of the function $h_e^s$ defined in \eqref{he}.
		\item In Section \ref{epis}, we prove Theorem \ref{thm1}, i.e. the epiperimetric inequality for $\Ws$.
		\item In Section \ref{epim}, we prove Theorem \ref{thm2}, i.e. the logarithmic epiperimetric inequality for $\Wm$.
		\item In Section \ref{epin}, we state and prove two epiperimetric inequalities for negative energies $\Ws$ and $\Wm$ (Theorem \ref{thm4} and Theorem \ref{thm5}).
		\item In Section \ref{fgap}, using the four epiperimetric inequalities proved in the previous sections, we establish a frequency gap in Proposition \ref{gap}.
		\item In Section \ref{char}, using Theorem \ref{thm1} and Theorem \ref{thm2}, we prove Proposition \ref{cha}, i.e. the characterization of the $\lambda$-homogeneous solutions of \eqref{fract4} (with 0 obstacle), in the case $\lambda=1+s$ and $\lambda=2m$.
		\item In Section \ref{final}, we use Theorem \ref{thm1} and Theorem \ref{thm2} to prove our main result on the regularity and the structure of the free boundary (Theorem \ref{thm3}).
	\end{itemize}
	\subsection*{Acknowledgment.} I would like to thank Bozhidar Velichkov for all the useful discussion and encouragement.
	
	The author was partially supported by the European Research Council (ERC), EU Horizon 2020 programme, through the project ERC VAREG - \textit{Variational approach to the regularity of the free boundaries} (No. 853404).
	\section{Preliminaries}\label{prel}

	\subsection{Generalized Almgren's frequency function} The original generalized Almgren's frequency function in \cite{css08} must be modified in the case $\vf\in C^{k,\gamma}(\R^n)$, as in the following proposition. 
	
	We follow \cite{gr17}, but a similar generalized Almgren's frequency can be found in \cite{css08} or in \cite{bfr18}.
	
	First we recall the function $H^{x_0}(r):=H^{x_0}(r,v)$ as in \eqref{H}, with $v=u^{x_0}$. We drop the dependence on $x_0$ if it is not ambiguous.
	\begin{proposition}[Monotonicity of generalized Almgren's frequency]\label{mono} Let $u$ be a solution of \eqref{fract2} and $v=u^{x_0}$ be the solution of \eqref{fract5} with $x_0\in\Gamma(u)$, $\vf\in C^{k,\gamma}(\R^n)$, $k\ge2$ and $\gamma\in(0,1)$. We define \bea\Phi^{x_0}(r,v):=(r+Cr^{1+p})\frac{d}{dr}\log\max\{H^{x_0}(r),r^{n+a+2(k+\gamma-p)}\},\eea for $p\in(0,\gamma)$ and $C$ large enough. Therefore, there is $r_0>0$ such that the function $r\mapsto \Phi^{x_0}(r,v)$ is monotone increasing for $r\in(0,r_0)$.
		\begin{proof}
			The proof is technical and we skip it, since it is standard and proved in many variations. We refer to Proposition 6.1 in \cite{gr17} for the complete proof.
		\end{proof}
	\end{proposition}
	
	By the previous proposition, if $v=u^{x_0}$, it is well-defined $$\Phi^{x_0}(0^+,v):=\lim_{r\to0^+}\Phi^{x_0}(r,v),$$ for $x_0\in\Gamma(u)$.
	
	In the case of obstacle $\vf\equiv0$, if the "pure" Almgren's frequency $N^{x_0}(0^+,u)=\lambda$, with $N^{x_0}$ as in \eqref{alm}, then we can consider a blow-up of $u$ at $x_0$, that is a $\lambda-$homogeneous solution (see \cite{acs08}).
	
	We want a similar result for $\vf\in C^{k,\gamma}(\R^n)$. Precisely, we will show that if $\Phi^{x_0}(0^+,v)=n+a+2\lambda$, then the blow-up of $v$ at $x_0$ is a $\lambda-$homogeneous solution. But this is only true if $\lambda<k+\gamma$, and this is the motivation for defining a new generalized Almgren's frequency, which was originally introduced in \cite{css08}.
	
	To prove this result, we need the following two lemmas from \cite{gr17}.
	\begin{lemma}\label{h1h2}
	Let $u$ be a solution of \eqref{fract2} and $v=u^{x_0}$ be the solution of \eqref{fract5} with $x_0\in\Gamma(u)$, $\vf\in C^{k,\gamma}(\R^n)$, $k\ge2$ and $\gamma\in(0,1)$. If $\Phi^{x_0}(0^+,v)=n+a+2\lambda$ with $\lambda<k+\gamma$, then \be\label{h1} H^{x_0}(r,v)\le Cr^{n+a+2\lambda}\ee for all $r\in(0,r_0).$ Moreover, for every $\eps>0$ there is $r_\eps$ such that \be\label{h2} H^{x_0}(r,v)\ge Cr^{n+a+2(\lambda+\eps)}\ee for all $r\in(0,r_\eps)$.
		In particular, $\Phi^{x_0}(0^+,v)$ does not depend on $p\in(0,\gamma)$.
		\begin{proof} 
			The proof of the first part follows by the monotonicity of generalized Almgren's frequency and an integration from $r$ to $r_0$. The second part is similar. See Lemma 6.4 in \cite{gr17} for more details.
		\end{proof}	
	\end{lemma}
	
	Now we proceed with the second lemma.
	\begin{lemma}
		Let $u$ be a solution of \eqref{fract2} and $v=u^{x_0}$ be the solution of \eqref{fract5} with $x_0\in\Gamma(u)$, $\vf\in C^{k,\gamma}(\R^n)$, $k\ge2$ and $\gamma\in(0,1)$. If $\Phi^{x_0}(0^+,v)=n+a+2\lambda$ with $\lambda<k+\gamma$, then \be\label{est} \left\lvert\int_{B_r(x_0)}vh\y\,d\HH^n\right\rvert\le Cr^{n+a+\lambda+k+\gamma-1}.\ee
		\begin{proof} 
			
			By \eqref{estimate}, we obtain \bea \left\lvert\int_{B_r(x_0)}vh\y\,d\HH^n\right\rvert&\le Cr^{k+\gamma-2}\int_{B_r}\lvert v\rvert\y\,dX\\& \le Cr^{k+\gamma-2}\left(\int_{B_r(x_0)}\y\,dX\right)^\frac12\left(\int_{B_r(x_0)}v^2\y\right)^\frac12\\&\le Cr^{k+\gamma-2}r^{\frac{n+a+1}{2}}r^{\frac{n+a+2\lambda+1}{2}}=Cr^{n+a+\lambda+k+\gamma-1},\eea where in the last inequality we used \eqref{h1}.
		\end{proof}	
	\end{lemma}
	
	Now we are ready to prove the existence of a $\lambda-$homogeneous blow-up, when $\lambda<k+\gamma$. We denote by \be\label{c1a}C^{1,\alpha}_a(\Omega):=\{w\in C^1(\Omega):\norm w_{C^{1,\alpha}_a(\Omega)}<+\infty\},\ee where \bea\norm w_{C^{1,\alpha}_a(\Omega)}:=\norm w_{C^{\alpha}(\Omega)}+\lVert \nabla_x w\rVert_{C^{\alpha}(\Omega)}+\lVert \y\partial _y w\rVert_{C^{\alpha}(\Omega)}.\eea
	\begin{proposition} \label{prop} Let $u$ be a solution of \eqref{fract2} and $v=u^{x_0}$ be the solution of \eqref{fract5} with $x_0\in\Gamma(u)$, $\vf\in C^{k,\gamma}(\R^n)$, $k\ge2$ and $\gamma\in(0,1)$. If $\Phi^{x_0}(0^+,v)=n+a+2\lambda$ with $\lambda<k+\gamma$, then, up to subsequences, the rescalings \be\label{rescaling}v_{r,x_0}(x,y):=\frac{v(x_0+rx,ry)}{\frac{1}{r^{n+a}}\int_{\partial B_r(x_0)} v ^2\y\,d\mathcal{H}^n}\ee converge in $C^{1,\alpha}_a(B_R^+)$ for all $R>0$ to a blow-up $v_{0,x_0}$, as $r\to0^+$, which is a solution of \eqref{fract4} (with 0 obstacle) and it is $\lambda-$homogeneous.
		\begin{proof}
			We can proceed as in the proof of Proposition 6.6 in \cite{gr17}.
			
			We want to point out that, since the values of $\Phi^{x_0}(0^+,v)$ do not depend on $p\in(0,\gamma)$, by Lemma \ref{h1h2}, we can choose $p\in(0,\gamma)$ small enough such that $$\lambda+\eps<k+\gamma-p,$$ for some $\eps>0$. Hence there is $r_1>0$ such that $H(r)\ge r^{n+a+2(k+\gamma-p)}$ for all $r\in(0,r_1)$, by \eqref{h2}.
			
			In particular, the monotonicity of the function $r\mapsto r(1+Cr^p)\frac{H'(r)}{H(r)}$ is equivalent to the monotonicity of the function $r\mapsto \Phi^{x_0}(r,v)$ for $r$ small enough.
		\end{proof}
	\end{proposition}

	Notice that the blow-up $v_{0,x_0}$ is not a priori unique. Still, all the blow-ups at $x_0$ have the same homogeneity. Therefore, for $\lambda<k+\gamma$, the following set is well-defined: 
	\be\label{gamma}\Gamma_\lambda(u):=\{x_0\in\Gamma(u):v_{0,x_0} \mbox{ is } \lambda-\mbox{homogeneous}\}.\ee
	Equivalently, $\Gamma_\lambda(u)$ is the set of all points $x_0\in\Gamma(u)$ such that $\Phi^{x_0}(0^+,v)=n+a+2\lambda$, with $\lambda<k+\gamma$.

	\begin{remark}\label{for}
		As in the proof of the last proposition, we get  \be \Phi^{x_0}(r,v)=(1+Cr^p)\left(n+a+2N^{x_0}(r,v)-2\frac{r}{H^{x_0}(r)}\int_{B_{Rr}(x_0)}vh\y\,dX\right),\ee for $r$ small enough. Then, by \eqref{h2} and \eqref{est}, we obtain $$\Phi^{x_0}(0^+,v)=n+a+2N^{x_0}(0^+,v).$$
	\end{remark}
	
	\subsection{Properties of the Weiss' energy $\W$} In the case $\vf\not\equiv0$, we can consider the following Weiss' energy, which is a small modification of the Weiss' energy $W$ for the obstacle $\vf\equiv0$, defined in \eqref{W}. Precisely, we define \bea\W^{x_0}_\lambda(r,v)=W^{x_0}_\lambda(r,v)+\frac{1}{r^{n+a+2\lambda-1}}\int_{B_r(x_0)}vh\y\,dX\eea and we drop the dependence on $x_0$ if $x_0=0$.
	
	\begin{lemma}\label{h'd'} Let $u$ be a solution of \eqref{fract2} and $v=u^{x_0}$ be the solution of \eqref{fract5} with $x_0\in\Gamma(u)$, $\vf\in C^{k,\gamma}(\R^n)$, $k\ge2$ and $\gamma\in(0,1)$. If $$\mathcal{I}(r):=\int_{\partial B_r(x_0)} v \partial_{\nu}v\y\,d\HH^n=\int_{B_r(x_0)} \lvert \nabla v \rvert^2\y\,dX+\int_{B_r(x_0)} vh\y\,dX,$$ then
		\begin{equation}\label{h'}
			H'(r)=\frac{n+a}{r}H(r)+2\mathcal{I}(r),
		\end{equation}
		and
		\begin{equation}\begin{aligned}\label{d'}
				\mathcal{I}'(r)&=\frac{n+a-1}{r}\mathcal{I}(r)-\frac{n+a-1}{r}\int_{B_r(x_0)} vh\y\,dX\\&\qquad+2\int_{\partial B_r(x_0)} (\partial_{\nu}v)^2\y\,d\mathcal{H}^n-2\int_{ B_r(x_0)}h\partial_{\nu}v\y\,dX\\&\qquad+\int_{\partial B_r(x_0)}vh\y\,d\HH^n.%=\\&=\frac{n+a-1}{r}\int_{B_r(x_0)}\lvert \nabla u\rvert^2\y\,dX+2\int_{\partial B_r(x_0)} (\partial_{\nu}u)^2\y\,d\mathcal{H}^n+\\&\qquad-2\int_{ B_r(x_0)}h\partial_{\nu}u\y\,dX
			\end{aligned}	
		\end{equation}
		\begin{proof}
			We give only the idea of the proof; for more details we refer to \cite{gpps17}.
			First we can change the variables in the expression of $H(r)$ and with a standard computation we get \eqref{h'}.
			
			Now, if $X=(x,y)$, then using $$\mbox{div}\left(\y\frac{\lvert\nabla v\rvert^2}{2}X-\y \nabla v(X\cdot \nabla v)\right)=\frac{n+a-1}{2}\y\lvert\nabla v\rvert^2-\y h(X\cdot \nabla v)$$ and an integration by parts, we get \eqref{d'}.
		\end{proof}		
	\end{lemma}

	We next show that also the Weiss' energy $\W$ satisfies a monotonicity formula.
	\begin{proposition}[Monotonicity of the Weiss' energy $\W$] Let $u$ be a solution of \eqref{fract2} and $v=u^{x_0}$ be the solution of \eqref{fract5} with $x_0\in\Gamma(u)$, $\vf\in C^{k,\gamma}(\R^n)$, $k\ge2$ and $\gamma\in(0,1)$. If $\lambda<k+\gamma$, then there exists contants $C,r_0>0$ such that for all $x_0\in\Gamma(u)$ and for all $r\in(0,r_0)$, it holds \be\label{monoW}\frac{d}{dr}\left(\W^{x_0}_\lambda(r,v)+Cr^{k+\gamma-\lambda}\right)\ge\frac{2}{r^{n+a+2\lambda+1}}\int_{\partial B_r(x_0)} (\nabla v\cdot x -\lambda v)^2\y\,d\mathcal{H}^n,\ee i.e. the function $r\mapsto\W^{x_0}_\lambda(r,v)+Cr^{k+\gamma-\lambda}$ is increasing for all $r\in(0,r_0)$.
		\begin{proof}
			The original proof is done in \cite{gpps17}, Proposition 3.5, in the case $\lambda=1+s$ and with a different exponent for $r$ in left-hand side. We generalize this result to any $\lambda$ and to any $\vf\in C^{k,\gamma}(\R^n)$.
			
			Since we have
			\bea \frac{d}{dr}\W^{x_0}_\lambda(r,v)&=-\frac{n+a+2\lambda-1}{r^{n+a+2\lambda}}\mathcal{I}(r)+\frac{\mathcal{I}'(r)}{r^{n+a+2\lambda-1}}+\lambda\frac{n+a+2\lambda}{r^{n+a+2\lambda+1}}H(r)\\&\qquad-\frac{\lambda}{r^{n+a+2\lambda}}H'(r),\eea by \eqref{h'} and \eqref{d'}, we obtain \bea\frac{d}{dr}\W^{x_0}_\lambda(r,v)&=\frac{2}{r^{n+a+2\lambda+1}}\int_{\partial B_r}(\nabla v\cdot x -\lambda v)^2\y\,d\mathcal{H}^n\\&-\frac{n+a-1}{r^{n+a+2\lambda}}\int_{B_r(x_0)} vh\y\,dX-\frac{2}{r^{n+a+2\lambda-1}}\int_{ B_r(x_0)}h\partial_{\nu}v\y\,dX\\&+\frac{1}{r^{n+a+2\lambda-1}}\int_{\partial B_r(x_0)}vh\y\,d\HH^n.\eea Now, we can use \eqref{est} to get \begin{equation*}\begin{aligned} & \left\lvert \frac{n+a-1}{r^{n+a+2\lambda}}\int_{B_r(x_0)} vh\y\,dX\right\rvert+\left\lvert\frac{2}{r^{n+a+2\lambda-1}}\int_{ B_r(x_0)}h\partial_{\nu}v\y\,dX\right\rvert\\&=\left\lvert\frac{1}{r^{n+a+2\lambda-1}}\int_{\partial B_r(x_0)}vh\y\,d\HH^n\right\rvert\le C r^{k+\gamma-\lambda-1} =C \frac{d}{dr}r^{k+\gamma-\lambda},
				\end{aligned}	
			\end{equation*} which proves the claim.
		\end{proof}	
	\end{proposition} 
	
	\begin{proposition}Let $u$ be a solution of \eqref{fract2} and $v=u^{x_0}$ be the solution of \eqref{fract5} with $x_0\in\Gamma(u)$, $\vf\in C^{k,\gamma}(\R^n)$, $k\ge2$ and $\gamma\in(0,1)$. If $\lambda<k+\gamma$ and $x_0\in\Gamma_{\lambda}(u)$, then \be\label{van}\W^{x_0}_\lambda(0^+,v)=0.\ee In particular, using \eqref{monoW}, we get \be\label{ric}\W^{x_0}_\lambda(r,v)\ge -Cr^{k+\gamma-\lambda}\ee for all $r\in(0,r_0)$.
		\begin{proof}
			We can write $$\W^{x_0}(r,v)=\frac{H^{x_0}(r)}{r^{n+a+2\lambda}}\left(r\frac{\mathcal{I}^{x_0}(r)}{H^{x_0}(r)}-\lambda\right).$$ 
			Hence, using \eqref{h2} and Remark \ref{for}, we get \bea \lim_{ r\to0^+}r\frac{\mathcal{I}^{x_0}(r,v)}{H^{x_0}(r)}&=\lim_{ r\to0^+}\left(N^{x_0}(r)-\frac{r}{H^{x_0}(r)}\int_{B_{r}(x_0)}vh\y\,dX\right)\\&=N^{x_0}(0^+,v)
			=\frac{\Phi^{x_0}(0^+,v)-(n+a)}2 =\lambda,\eea 
			which concludes the proof.
		\end{proof}
	\end{proposition}
	
	\subsection{Homogeneous rescalings and homogeneous blow-up}
	The sequence \eqref{rescaling} has good rescaling properties with respect to the Almgren's frequency function.
	Now we consider another sequence of rescalings, the homogeneous rescalings, which have good rescaling properties with respect to the Weiss' energy.
	\begin{proposition}\label{prop2} Let $u$ be a solution of \eqref{fract2} and $v=u^{x_0}$ be the solution of \eqref{fract5} with $x_0\in\Gamma(u)$, $\vf\in C^{k,\gamma}(\R^n)$, $k\ge2$ and $\gamma\in(0,1)$. Suppose that $\Phi^{x_0}(0^+,v)=n+a+2\lambda$ with $\lambda<k+\gamma$. Let
	$v^{(\lambda)}_{r,x_0}$ be the homogeneous rescalings of $v$ at $x_0\in\Gamma(u)$, defined as \be \label{hom}v^{(\lambda)}_{r,x_0}(x,y):=\frac{v(x_0+rx,ry)}{r^\lambda}.\ee Then, up to a subsequence, the homogeneous rescalings converge in $C^{1,\alpha}_a(B_R^+)$ for all $R>0$ (defined in \eqref{c1a}), as $r\to0^+$, to a  blow-up $v^{(\lambda)}_{0,x_0}$, which is $\lambda-$homogeneous and is a solution to \eqref{fract4} (with 0 obstacle).
		\begin{proof}		
			By the Poincaré inequality \eqref{poinc}, in order to show the boundness of $v_r$ in $H^1(B_R,a)$, it is sufficient to prove the boundness of $v_r$ in $L^2(\partial B_R,a)$ and the boundness of $\lvert \nabla v_r\rvert $ in $L^2(B_R,a)$. The first bound follows by \eqref{h1}. In fact 
			$$\lVert v_r\rVert^2_{L^2(\partial B_R,a)}=H(R,v_r)=\frac{H^{x_0}(Rr,v)}{r^{n+a+2\lambda}}\le C(R).$$ 
			Also the second bound follows from \eqref{h1}:
			\bea\int_{B_R}\lvert \nabla v_r\rvert^2 \y\,dX&=\frac{1}{r^{n+a+2\lambda-1}}\int_{B_{Rr}(x_0)}\lvert \nabla v \rvert^2\y\,dX\\&\le C(R) \frac{Rr}{H^{x_0}(Rr)}\int_{B_{Rr}(x_0)}\lvert \nabla v\rvert^2 \y\,dX\\&=C(R)\Biggl( N^{x_0}(Rr)-\frac{Rr}{H^{x_0}(Rr)}\int_{B_{Rr}(x_0)}vh\y\,dX\\&\qquad+\frac{Rr}{H^{x_0}(Rr)}\int_{B_{Rr}(x_0)}vh\y\,dX\Biggl)\le C(R),
			\eea
			
			where in the last inequality we used  \eqref{h2}, \eqref{est} and the monotonicity of the function $r\mapsto\Phi^{x_0}(r,v)$. Thus, for $r$ small enough, $$r\frac{\mathcal{I}(r)}{H(r)}\le C, $$ as in the proof of Proposition \ref{prop}.
			
			Hence, up to subsequences, $v_r$ converges to some $v_0$ weakly in $H^1(B_R,a)$ and
			\be \label{la}\begin{aligned} L_a(v_r(x,y))&=\frac{r^{2}}{r^\lambda}L_av(x_0+rx,ry)\le\frac{r^{2}}{r^\lambda}r^{k+\gamma-2}\lvert x\rvert^{k+\gamma-2}\\&=r^{k+\gamma-\lambda}\lvert x\rvert^{k+\gamma-2}
			\end{aligned}
			\ee by \eqref{est}, with an equality in $\R^{n+1}\setminus(\{v_r=0\}\cap\{y=0\})$. Therefore, we can use the estimate in \cite{css08} (Proposition 4.3 and Lemma 4.4) to get the convergence in $C^{1,\alpha}_a$, since $v_r$ are solutions of \eqref{fract5}, where the right-hand sides in the third and the fourth lines are as in \eqref{la}.
			Moreover, $v_0$ is a solution of \eqref{fract4} (with 0 obstacle), since we can send $r\to0^+$ in \eqref{la}.
			
			Finally, we show that $v_0$ is homogeneous. Indeed, by \eqref{monoW}, we have that for all $0<R_1<R_2<r_0$ \begingroup\allowdisplaybreaks\begin{align*} W_\lambda(R_2r_k,v)&+C(R_2r_k)^{k+\gamma-\lambda}-W_\lambda(R_1r_k,v)-C(R_1r_k)^{k+\gamma-\lambda}\\&\ge \int_{R_1r_k}^{R_2r_k}\frac2{r^{n+a+2\lambda+1}}\int_{\partial B_r(x_0)}(\nabla v\cdot x -\lambda v)^2\y\,d\mathcal{H}^n\,dr\\&=\int_{R_1}^{R_2}\frac2{r^{n+a+2\lambda+1}r_k^{n+a+2\lambda}}\int_{\partial B_{rr_k}(x_0)}(\nabla v\cdot x -\lambda v)^2\y\,d\mathcal{H}^n\,dr\\&=\int_{R_1}^{R_2}\frac2{r^{n+a+2\lambda+1}}\int_{\partial B_{r}}(\nabla v_r\cdot x -\lambda v_r)^2\y\,d\mathcal{H}^n\,dr,
			\end{align*}
			\endgroup where sending $r_k\to0^+$ the left-hand side vanishes, by \eqref{van}. Thus, by arbitrariness of $R_1$ and $R_2$, we obtain that $v_0$ is homogeneous, since $\nabla v_0\cdot x=\lambda v_0$ on $\partial B_r$ for all $r\in(0,r_0)$.
		\end{proof} 
	\end{proposition} 
	
	\subsection{The operator $L_a^S$}
	The strategy to prove the epiperimetric inequalities is to decompose a trace $c\in H^1(\partial B_1,a)$ in terms of the eigenfunctions of the operator $L_a$ restricted to $\partial B_1$. The restriction $L_a^S$ is defined for any function $\phi\in H^1(\partial B_1,a)$ as \be\label{las}L_a^S\phi=\lvert y\rvert^{-a}\mbox{div}(\y\nabla \Phi)|_{\lvert x \rvert=1},\ee where $\Phi(x)=\phi\left(\frac{x}{\lvert x \rvert}\right)$.
	\begin{remark}\label{spherical}
		Since in spherical coordinates we have \be\label{sphh}\mbox{div}(\y\nabla u)= \y\frac{\partial^2 u}{\partial r^2}+\frac{1}{r}(n+a)\y\frac{\partial u}{\partial r}+\frac{1}{r^2}\y L_a^S u,\ee then  $$L_a^S(r^\alpha\phi(\theta))=0$$ if and only if $$-L_a^S\phi(\theta)=\lambda^a(\alpha)\phi(\theta),$$ where $\lambda^a(\alpha)=\alpha(\alpha+n+a-1)$. 
		
		By Liouville Theorem \ref{liouville}, if we suppose that $\phi$ is even in the $y$ direction, then $r^\alpha \phi(\theta)$ is $L_a-$harmonic if and only if $r^\alpha \phi(\theta)$ is a polynomial $L_a-$harmonic with $\alpha\in\N$.

		Using the theory of compact operators, we can prove that there exist an increasing sequence of eigenvalues
		$\{\lambda^a_k\}_{k\in\mathbb{N}}\subset \mathbb{R}_{\ge0}$ and a sequence of eigenfunctions $\{\phi_k\}_{k\in\mathbb{N}}\subset H^1(\partial B_1,a)$ normalized in $L^2(\partial B_1,a)$, such that $$-L_a^S\phi_k=\lambda^a_k \phi_k,$$ with $\{\phi_k\}_{k\in\mathbb{N}}$ orthonormal basis of $H^1(\partial B_1,a)$.
		
		The (normalized) eigenspace corresponding to eigenvalue $\lambda^a$ is $$E(\lambda^a)=\{\phi\in H^1(\partial B_1,a): -L_a^S \phi=\lambda^a \phi, \ \lVert \phi\rVert_{L^2({\partial B_1,a})}=1\}.$$ for all $\lambda^a\subset\{\lambda_k^a\}_{k\in\N}$.
		
		We denote by $\alpha_k\in\N$ the grade of the polynomial that corresponds to the eigenvalue $\lambda^a_k$, i.e. the only natural number such that $\lambda^a(\alpha_k)=\lambda^a_k$.
		
		In particular $ \lambda^a_1=\lambda^a(0)=0$ and $E(\lambda^a_1)$ is the space of constant functions, while
		$ \lambda^a_2=\ldots=\lambda^a_{n+2}=\lambda^a(1)=n+a$ and $E(\lambda^a_2)$ is the space of linear functions. Finally $\lambda^a_k \ge\lambda^a(2)$ for $k\ge n+3$.
	\end{remark}
	\subsection{The Weiss' energy $W$ and eigenfunctions of $-L_a^s$} The following lemma is a generalization of Lemma 2.3 and Lemma 2.4 in \cite{csv17}, for the Weiss' energy $W$ with weight. It will be used in several proof later.
	\begin{lemma} \label{spherical1} Let $\phi\in H^1(\partial B_1,a)$ with $$\phi(\theta)=\sum_{k=1}^{\infty} c_k \phi_k(\theta)\in H^1(\partial B_1,a),$$ where $\phi_k$ normalized eigenfunctions of $-L_a^S$ as above, and let $r^{\mu}\phi(\theta)$ be the $\mu-$homogeneous extension, then 
		\begin{equation}\label{prima}
			W_\mu(r^\mu \phi)=\frac{1}{n+a+2\mu-1}\sum_{k=1}^{\infty}\left(\lambda_k^a-\lambda^a(\mu)\right)c_k^2.
		\end{equation}
		Moreover, if \be\label{kappa}\kappa_{\alpha,\mu}^a=\frac{\alpha-\mu}{\alpha+\mu+n+a-1},\ee then \be\label{seconda} W_\mu(r^\alpha\phi)-(1\mp\kappa_{\alpha,\mu}^a)W_\mu(r^\mu\phi)=\frac{\pm\kappa_{\alpha,\mu}^a}{n+a+2\alpha-1}\sum_{k=1}^{\infty}(\lambda^a(\alpha)-\lambda^a_k)c_k^2.\ee
		Finally, if $c\in H^1(\partial B_1)$ such that $r^{\mu+t}c$ is a solution of \eqref{fract4} (with 0 obstacle), then 
		\begin{equation}\label{terza}
			W_\mu(r^{\mu+t}c)=t\norm c_{L^2(\partial B_1,a)}^2
		\end{equation}
		and
		\begin{equation}\label{quarta}
			W_\mu(r^{\mu}c)=\left(1+\frac{t}{n+a+2\mu-1}\right)W_\mu(r^{\mu+t}c).
		\end{equation}
		\begin{proof} The proof is very similar to the one in \cite{csv17}, but we briefly recall it for the sake of completeness.
			
			By \eqref{w1} we get \bea W_\mu(r^\alpha \phi)=\sum_{k=1}^\infty c_k^2\left(\frac{\lambda_k^a+\alpha^2}{n+a+2\alpha-1}-\mu\right),\eea from where \eqref{prima} and \eqref{seconda} follow.

			Now if $r^{\mu+t}c$ is a solution, then $W_{\mu+t}(r^{\mu+t}c)=0$, by \eqref{w4}. Therefore \eqref{terza} holds.
			
			Finally, the proof of \eqref{quarta} follows by \eqref{w2}, \eqref{w3} and \eqref{terza}.
		\end{proof}
	\end{lemma}

	\subsection{Properties of $h_e^s$}
	Finally we recall the properties of the function $h_e^s$ defined in \eqref{he}, which is the only $(1+s)-$homogeneous solution of \eqref{fract4} (with 0 obstacle). The latter follows from Proposition \ref{cha}, but it was originally proved in \cite{css08}.
	\begin{proposition} \label{heprop} Let $e\in\partial B_1'$ and $h_e^s$ as in \eqref{he}, then
		$h_e^s$ is a $(1+s)-$homogeneous solution of \eqref{fract4} (with 0 obstacle), $h_e^s=0$ on $B'_1\cap\{x\cdot e\le0\}$ and it holds
		\begin{equation*}
			\lim_{y\to0^+}(\y\partial_yh_e^s)=
			\begin{cases} 
				-c_s\lvert x\cdot e \rvert^{1-s} & B'_1\cap\{x\cdot e <0\}\\ 
				0 & B'_1\cap\{x\cdot e \ge0\},
			\end{cases}
		\end{equation*} with $c_s=2^{1-s}(1+s)$.
		Finally, the $L^2(\partial B_1,a)$ projection on linear functions of $h_e^s$ has the form $C(x\cdot e)$ for some $C>0$.
		\begin{proof} The proof is a simple computation, hence it will be omitted.
		\end{proof}
	\end{proposition}
	\section{Epiperimetric inequality for $\Ws$}\label{epis}
	The proof of the epiperimetric inequality for $\Ws$ follows the ideas of the proof from \cite{csv17} in the case $s=\frac12$, i.e. we decompose the trace $c\in H^1(\partial B_1,a)$ in terms of eigenfunction of $L_a^S$. 
	\subsection{Decomposition of $c$}
	Let $c\in H^1(\partial B_1,a)$ even in the $y$ direction and such that $c\ge0$ in $\partial B'_1$. We decompose $c$ using eigenfunctions of the operator $-L_a^S$ defined in \eqref{las}.
	
	The projection on linear functions $E(\lambda_2^a)$ of $c$ has the form $c_1(x\cdot e)$ for some $e\in\partial B'_1$, then the projection of $h_e^s$ on $E(\lambda_2^a)$ has the same form $C(x\cdot e)$ for $C>0$. Thus we can choose $C>0$ such that $c$ and $Ch_e^s$ has the same projection on $E(\lambda_2^a)$.
	
	Notice that the function $u_0(x,y)=\lvert y \rvert^{1+s}$ restricted on $\partial B_1$ has 0 projection on $E(\lambda_2^a)$. Therefore we can choose $c_0\in\R$ such that the projections of $c-Ch_e^s$ and $u_0$ on the constant functions $E(\lambda^a_1) $ are the same. Then $$c(\theta)=Ch_e^s(\theta)+c_0u_0(\theta)+\phi(\theta),$$ where \be\label{perd2}\phi(\theta)=\sum_{\{k:\ \lambda^a_k\ge\lambda^a(2)\}} c_k\phi_k(\theta).\ee Hence we can decompose $z$ as $$z(r,\theta)=Cr^{1+s}h_e^s(\theta)+c_0r^{1+s}u_0(\theta)+r^{1+s}\phi(\theta)$$ and we can define the competitor $\zeta$ as $$\zeta(r,\theta)=Cr^{1+s}h_e^s(\theta)+c_0r^{1+s}u_0(\theta)+r^{2}\phi(\theta),$$ which is an admissible function ($\zeta\in\mathcal{K}_c$), since $\zeta(r,\theta)\ge r^2c(\theta)$ on $B'_1$.
	\subsection{Proof of Theorem \ref{thm1}}
	Let's start with a lemma, that will allow us to compute the Weiss' energy $\Ws$ of $z$ and $\zeta$.
	\begin{lemma} If $\psi=r^\alpha\phi(\theta)$, with $\phi\in H^1(\partial B_1,a)$, then \bea W_{1+s}(Ch_e^s+cu_0+\psi)&=-c_0^2(1+s)(1-s)\int_{B_1}\lvert y\rvert \,dX+W_{1+s}(\psi)\\&\qquad+\frac{1}{n+\alpha+1-s}\beta(\phi),\eea where \bea \beta(\phi)&:=-2c_0(1+s)(1-s)\int_{\partial B_1} \phi(\theta)\lvert \theta_{n+1}\rvert^{-s}\,d\mathcal{H}^{n}\\&\qquad+4c_sC\int_{\partial B'_1} \phi(\theta')(\theta'\cdot e)_-^{1-s}\,d\mathcal{H}^{n-1}.\eea
		\begin{proof}
			By Proposition \ref{heprop} and \eqref{w4}, we obtain
			\begingroup
			\allowdisplaybreaks
			\begin{align*}
				&W_{1+s}(Ch_e^s+c_0u_0+\psi)=C^2W_{1+s}(h_e^s)+W_{1+s}(c_0u_0+\psi)\\&\qquad+2C\biggl(\int_{B_1}\nabla h_e^s\cdot \nabla(c_0u_0+\psi)\y\,dX\\&\qquad-(1+s)\int_{\partial B_1}h_e^s(c_0u_0+\psi)\y\,d\mathcal{H}^n \biggl)\\&=W_{1+s}(c_0u_0+\psi)+2C\left(\int_{B_1}\nabla h_e^s\cdot \nabla \psi\y\,dX-({1+s})\int_{\partial B_1}h_e^s\psi\y\,d\mathcal{H}^n\right)\\&=W_{1+s}(c_0u_0+\psi)+4c_sC\left(\int_{B'_1}\psi(x,0) (x\cdot e)_-^{1-s}\,d\mathcal{H}^n\right)\\&=W_{1+s}(cu_0+\psi)+4c_sC \int_0^1\int_{\partial B'_1} r^\alpha \phi(\theta')r^{1-s}(\theta'\cdot e)_-^{1-s}r^{n-1}\,d\mathcal{H}^{n-1}(\theta')\,dr\\&=W_{1+s}(c_0u_0+\psi)+\frac{4c_sC}{n+\alpha+1-s}\int_{\partial B'_1} \phi(\theta')(\theta'\cdot e)_-^{1-s}\,d\mathcal{H}^{n-1}(\theta')
			\end{align*}
			\endgroup
			Additionally, since $u_0$ is $(1+s)-$homogeneous and $$L_au_0\y=(1+s)(1-s)\lvert y\rvert ^{-s},$$ we have that
			\begingroup
			\allowdisplaybreaks
			\begin{align*}
				W_{1+s}(cu_0&+\psi)=c_0^2W_{1+s}(u_0)+W_{1+s}(\psi)\\&\qquad+2c_0\left(\int_{B_1}\nabla u_0\cdot \nabla\psi\y\,dX-({1+s})\int_{\partial B_1}u_0\psi\y\,d\mathcal{H}^n\right)\\&=c_0^2W_{1+s}(u_0)+W_{1+s}(\psi)-2c_0\int_{B_1}\psi L_a u_0\y\,dX\\&=c_0^2W_{1+s}(u_0)+W_{1+s}(\psi)-2c_0(1+s)(1-s)\int_{B_1}\psi\lvert y \rvert^{-s}\,dX \\&=c_0^2W_{1+s}(u_0)+W_{1+s}(\psi)\\&\qquad-2c_0(1-s)({1+s})\int_0^1\int_{\partial B_1} r^\alpha \phi(\theta)r^{-s}\lvert \theta_{n+1}\rvert^{-s}r^{n}\,d\mathcal{H}^{n}(\theta)\,dr\\&=c_0^2W_{1+s}(u_0)+W_{1+s}(\psi)\\&\qquad-\frac{2c_0(1+s)(1-s)}{n+\alpha+1-s}\int_{\partial B_1} \phi(\theta)\lvert \theta_{n+1}\rvert^{-s}\,d\mathcal{H}^{n}(\theta).
			\end{align*}
			\endgroup
			Finally, we notice that \be\label{perd1}W_{1+s}(u_0)=-\int_{B_1}u_0L_a u_0\y\,dX=-(1+s)(1-s)\int_{B_1}\lvert y\rvert^{1+s}\lvert y\rvert^{-s}\,dX<0,\ee which, together with the previous identities, gives the claim.
		\end{proof}
	\end{lemma}
	\begin{proof}[Proof of Theorem \ref{thm1}] 
		By this lemma, we can easily conclude the proof of the epiperimetric inequality for $\Ws$. Indeed, if $\kappa=\kappa^a_{2,1+s}=\frac{1+a}{2n+a+5}$, we can use $$\frac{1}{n+2+1-s}-\left(1-\frac{1+a}{2n+a+5}\right)\left(\frac{1}{n+(1+s)+1-s}\right)=0,$$ with
		\eqref{seconda}, \eqref{perd2} and \eqref{perd1} to conclude.
	\end{proof}
	\begin{remark}\label{epi2} If the equality in the epiperimetric inequality holds, then $c_0=0$ and, by \eqref{seconda}, $\phi$ is an eigenfunction corresponding to the eigenvalue $\lambda^a(2)$. Therefore $$z(r,\theta)=Cr^{1+s}h_e^s(\theta)+r^{1+s}\phi(\theta).$$
		
		Furthermore, since $z\ge0$ on $B'_1$ and $h_e^s=0$ on $B'_1\cap \{x\cdot e<0\}$, we have that $r^{1+s}\phi\ge0$ on $B'_1\cap \{x\cdot e<0\}$, but $r^{1+s}\phi$ is even in the $y$ direction, so $r^{1+s}\phi\ge0$ on $B'_1$.
	\end{remark}
	
	\section{Logarithmic epiperimetric inequality for $\Wm$}\label{epim}
	
	The proof of the logarithmic epiperimetric inequality for $\Wm$ follows the ideas of the proof from \cite{csv17} in the case $s=\frac12$. The strategy is the same as the one of the proof of Theorem \ref{thm1}.
	\subsection{Construction of $h_{2m}$}
	For the proof of the logarithmic epiperimetric inequality, we need to build an eigenfunction of $-L_a^S$ as follows. 
	\begin{remark}\label{h2m} There is a $2m-$homogeneous $L_a-$harmonic polynomial $h_{2m}$ such that $h_{2m}\equiv 1$ on $\partial B'_1$.
		
		The polynomial is given by $$h_{2m}=\sum_{k=0}^m C_k y^{2k}(x_1^2+\ldots+x_n^2)^{m-k},$$ where the constants $C_k$ are yet to be chosen.
		
		Notice that 
		$$
		\begin{aligned}\sum_{i=1}^{n}\partial_{i,i} h_{2m} &=\sum_{k=0}^{m-1}4C_k(m-k)(m-k-1)y^{2k}(x_1^2+\ldots+x_n^2)^{m-k-1}\\&\qquad+\sum_{k=0}^{m-1}2nC_k(m-k)y^{2k}(x_1^2+\ldots+x_n^2)^{m-k-1}\\&=\sum_{k=1}^{m}C_{k-1}y^{2k-2}2(m-k+1)(2(m-k)+n)(x_1^2+\ldots+x_n^2)^{m-k}, \end{aligned}
		$$
		$$\partial_{y,y} h_{2m}=\sum_{k=1}^{m}C_k(2k)(2k-1)y^{2k-2}(x_1^2+\ldots+x_n^2)^{m-k} $$ and $$\frac ay\partial_y h_{2m} =\sum_{k=1}^maC_k(2k)y^{2k-2}(x_1+\ldots x_n)^{m-k}.$$
		Thus, $h_{2m}$ is $L_a-$harmonic if and only if $$C_k(2k)(2k-1+a)+2(m-k+1)(2(m-k)+n)C_{k-1}=0.$$
		Therefore, we can choose $$C_{k}=-\frac{2(m-k+1)(2(m-k)+n)}{2k(2k-1+a)}C_{k-1},$$ for $k\in\{1,\ldots m\}$ and $C_0=1$, which concludes the construction of $h_{2m}$.
	\end{remark}
	
	\subsection{Decomposition of $c$}
	
	Let $c\in H^1(\partial B_1,a)$, we can decompose $$c(\theta)=\sum_{k=1}^\infty c_k \phi_k(\theta), $$ where $\phi_k(\theta)$ are the normalized eigenfunctions of $-L_a^S$ with eigenvalues $\lambda_k^a$ and corresponding homogeneity $\alpha_k$, then $$c(\theta)=P(\theta)+\phi(\theta),$$ with $$P(\theta)=\sum_{\{k:\ \alpha_k\le2m\}}c_k \phi_k(\theta)$$ and $$\phi(\theta)=\sum_{\{k:\ \alpha_k>2m\}}c_k \phi_k(\theta).$$ 
	
	Let $z=r^{2m}c$ be the $2m-$homogeneous extension of $c$ and let $h_{2m}$ as in Remark \ref{h2m}, therefore $$z(r,\theta)=r^{2m}P(\theta)+Mr^{2m}h_{2m}(\theta)-Mr^{2m}h_{2m}(\theta)+r^{2m}\phi(\theta),$$ where $M=\max\{P_-(\theta'):\theta'\in\partial B'_1\}$ and $P_-(\theta)$ is the negative part of $P(\theta)$.
	
	We choose a competitor $\zeta$ extending with homogeneity $\alpha>2m$ the high modes on the sphere and leaving the rest unchanged, i.e. $$\zeta(r,\theta)=r^{2m}P(\theta)+Mr^{2m}h_{2m}(\theta)-Mr^{\alpha}h_{2m}(\theta)+r^{\alpha}\phi(\theta),$$for some $2m<\alpha<2m+\frac12$, then $$\zeta(r,\theta)\ge r^\alpha c(\theta)\ge0$$ on $B'_1$, since we have chosen $M>0$ such that $P(\theta)+M\ge0$ $\forall \theta\in \partial B'_1$. This means that $\zeta\in\mathcal{K}_c$.
	
	Defining $\kappa_{\alpha,2m}^a$ as in \eqref{kappa}, we will choose $\varepsilon =\varepsilon (n,m)>0$ small enough and $2m<\alpha<2m+\frac12$ such that \begin{equation}\label{kkkk}\kappa_{\alpha,2m}^a=\varepsilon\Theta^{-\beta} \lVert\nabla_\theta \phi\rVert_{L^2(\partial B_1,a)}^{2\beta}
	\end{equation} for $\Theta>0$ and $\beta=\frac{n-1}{n+1}$. Notice that to be able to choose such $\alpha$, we must have an estimate of the type \begin{equation}\label{stimaaaa}
		\lVert\nabla_\theta \phi\rVert_{L^2(\partial B_1,a)}^{2}\le C_{n,m,a}\Theta,
	\end{equation} for some $C_{n,m}>0$, that should depend only on $n$, $m$ and $a$, since we want that $\alpha$ depends only on $n$, $m$ and $a$. For this reason we ask for the bounds $\norm c_{L^2(\partial B_1),a}^2\le\Theta$ and $\lvert W_{2m}(z)\rvert\le\Theta$.
	\subsection{Proof of Theorem \ref{thm2}}
	First we want to compute the term $W_{2m}(\zeta)-(1-\kappa_{\alpha,2m}^a)W_{2m}(z)$, and we see that for $\alpha$ near $2m$ and $\alpha>2m$, it is negative. This is contained in the following lemmas.
	\begin{lemma} In the hypotheses of Theorem \ref{thm2}, we have\begin{equation}\label{log1}
			W_{2m}(\zeta)-(1-\kappa_{\alpha,2m}^a)W_{2m}(z)\le C_1M^2 (\kappa^a_{\alpha,2m})^2-C_2\kappa^a_{\alpha,2m} \lVert\nabla_\theta \phi\rVert_{L^2(\partial B_1,a)}^{2},
		\end{equation} for some $ C_1 ,C_2>0$ depending only on $n$, $m$ and $a$.
		\begin{proof}
			We introduce the following functions: $$\psi(r,\theta)=\sum_{\{k:\ \alpha_k<2m\}} c_k r^{2m} \phi_k(\theta),$$ $$ H_{2m}(r,\theta)=Mr^{2m}h_{2m}(\theta)+\sum_{\{k:\ \alpha_k=2m\}} c_k r^{2m} \phi_k(\theta)$$ and $$\varphi_\mu(r,\theta)=-Mr^\mu h_{2m}(\theta)+\sum_{\{k:\ \alpha_k>2m\}} c_k r^{\mu} \phi_k(\theta). $$
			Then, we can write $z=\psi+H_{2m}+\varphi_{2m}$ and $\zeta=\psi+H_{2m}+\varphi_\alpha$.
			
			With this decomposition, $\psi$ is orthogonal in $L^2(B_1,a)$ and in $H^1(B_1,a)$ to $\varphi_\mu$ for $\mu=\alpha,2m$. Additionally, $H_{2m}$ is $L_a-$harmonic and $2m-$homogeneous, therefore using \eqref{w4}, we get $$W_{2m}(z)=W_{2m}(\psi+\varphi_{2m})=W_{2m}(\psi)+W_{2m}(\varphi_{2m})$$ and $$W_{2m}(\zeta)=W_{2m}(\psi+\varphi_{\alpha})=W_{2m}(\psi)+W_{2m}(\varphi_{\alpha}).$$ 
			Notice now that, since $\psi$ has only frequencies lower than $2m$, we have $W_{2m}(\psi)<0$. Thus, using \eqref{prima}, we get
			\begin{equation*}
				\begin{aligned}
					W_{2m}(\zeta)-(1-&\kappa^a_{\alpha,2m})W_{2m}(z)\\&=\kappa^a_{\alpha,2m}W_{2m}(\psi)+W_{2m}(\varphi_\alpha)-(1-\kappa^a_{\alpha,2m})W_{2m}(\varphi_{2m})\\&\le W_{2m}(\varphi_\alpha)-(1-\kappa^a_{\alpha,2m})W_{2m}(\varphi_{2m})\\&=M^2\lVert h_{2m}\rVert_{L^2(\partial B_1,a)}^2 \frac{\kappa^a_{\alpha,2m}}{n+a+2\alpha-1}(\lambda^a(\alpha)-\lambda^a(2m))\\&\qquad-\frac{\kappa^a_{\alpha,2m}}{n+a+2\alpha-1}\sum_{\{k:\ \alpha_k>2m\}}(\lambda_k^a-\lambda^a(\alpha))c_k^2\\&=M^2\lVert h_{2m}\rVert_{L^2(\partial B_1,a)}^2(\kappa^a_{\alpha,2m})^2\frac{(\alpha+2m+n+a-1)^2}{n+a+2\alpha-1} \\&\qquad-\frac{\kappa^a_{\alpha,2m}}{n+a+2\alpha-1}\sum_{\{k:\ \alpha_k>2m\}}(\lambda_k-\lambda(\alpha))c_k^2\\&\le C_1M^2(\kappa^a_{\alpha,2m})^2-\overline C\kappa^a_{\alpha,2m}\sum_{\{k:\ \alpha_k>2m\}}(\lambda_k-\lambda(\alpha))c_k^2,
				\end{aligned} 
			\end{equation*}where in the second equality we used \eqref{seconda} with $C_1$ and $\overline C$ depending only on $n$, $m$ and $a$.
			
			Observe that \begin{equation}\label{ult}\begin{aligned}
					\sum_{\{k:\ \alpha_k>2m\}} (\lambda^a_k-\lambda^a(\alpha))c_k^2&=\sum_{\{k:\ \alpha_k>2m\}} \lambda^a_kc_k^2-\lambda^a(\alpha)\sum_{\{k:\ \alpha_k>2m\}} c_k^2\\&\ge\sum_{\{k:\ \alpha_k>2m\}} \lambda^a_kc_k^2-\frac{\lambda^a(\alpha)}{\lambda^a(2m+1)}\sum_{\{k:\ \alpha_k>2m\}} \lambda^a_kc_k^2\\&\ge\sum_{\{k:\ \alpha_k>2m\}} \lambda^a_kc_k^2-\frac{\lambda^a(2m+\frac12)}{\lambda^a(2m+1)}\sum_{\{k:\ \alpha_k>2m\}} \lambda^a_kc_k^2\\&\ge \overline C_2\sum_{\{k:\ \alpha_k>2m\}}\lambda^a_kc_k^2=\overline C_2\lVert\nabla_\theta \phi\rVert_{L^2(\partial B_1,a)}^{2},
				\end{aligned}
			\end{equation} where we used $$\lVert \nabla _\theta\phi_k\rVert^2\ab=\lambda^a_k\lVert \phi_k\rVert^2\ab $$ and we have chosen $$\overline C_2=\frac{\lambda^a(2m+1)-\lambda^a(2m+\frac12)}{\lambda^a(2m+1)}.$$ Hence, we conclude by choosing $C_2=\overline C \overline C_2$ and $C_1$ as above.
		\end{proof}
	\end{lemma}
	\begin{lemma}
		In the same hypotheses of the Theorem \ref{thm2}, we have \begin{equation}\label{log2}
			M^2\le C_3\Theta^\beta\lVert\nabla_\theta \phi\rVert_{L^2(\partial B_1,a)}^{2(1-\beta)},
		\end{equation} for some $C_3>0$ that depends only on $n$, $m$ and $a$.
		\begin{proof}
			Since $\phi_k$ are $C^1$ on $\partial{B_1}$ and only a finite number of $k$ is such that $\alpha_k\le2m$, we have $$\lVert\nabla_\theta \phi_k\rVert_{L^\infty(\partial B_1)}^{2}\le L_m $$ for all $k$ such that $\alpha_k\le2m$ and for some $L_m>0$. 
			
			Moreover, the coefficients $c_k$ corresponding to $P$ are bounded by $\Theta^\frac12$, since $\norm P_{L^2(\partial B_1,a)}^2\le\norm c_{L^2(\partial B_1,a)}^2\le\Theta$, we get $P$ is $L-$Lipschitz continuous on $\partial{B_1}$, with $L=L_{n,m,a}\Theta^\frac12$.
			
			Now, since $\phi^2(\theta)\ge P_-^2(\theta)$ for all $\theta\in\partial B'_1$, it follows that
			$$\begin{aligned}
				\int_{\partial B'_1}& P_-^2\,d\mathcal{H}^{n-1}\le\int_{\partial B'_1} \phi^2\,d\mathcal{H}^{n-1}\\&\le C\left(\int_{\partial B_1} \lvert\nabla_\theta\phi\rvert^2\y\,d\mathcal{H}^n+\int_{ \partial B_1} \phi^2\y\,d\mathcal{H}^n\right)\\&\le C\left(\int_{\partial B_1} \lvert\nabla_\theta\phi\rvert^2\y\,d\mathcal{H}^n+\sum_{\{k: \ \alpha_k>2m\}}\frac{1}{\lambda^a_k}\int_{ \partial B_1} \lvert\nabla_\theta\phi_k\rvert^2\y\,d\mathcal{H}^n\right) \\&\le C\left(1+\frac{1}{\lambda^a(2m)}\right)\int_{ \partial B_1} \lvert\nabla_\theta\phi\rvert^2\y\,d\mathcal{H}^n,
			\end{aligned}$$by Proposition \ref{embedding}, and since $\phi$ contains only eigenfunctions corresponding to eigenvalue $\lambda^a_k>\lambda^a(2m)$.
			
			Finally we claim that $$\int_{\partial B_1'}P_-^2(\theta)\,d\mathcal{H}^{n-1}\ge C M^2 \left(\frac{M}{L}\right)^{n-1}=C\frac{{M}^{n+1}}{{L}^{n-1}},$$ in fact the norm in $L^2$ of $P_-^2(\theta)$ is controlled by the volume of an $(n-1)$-dimensional cone with height $M^2$ and radius of the base $\frac{M}{L}$, since the graph of $P_-$ must be above this cone, by the Lipschitz continuity of $P_-$. %(see Figure \ref{veli}). 
			
			Thus, since $1-\beta=\frac2{n+1}$ and $L=L_{n,m}\Theta^\frac12$, we obtain $$M^2\le C L^{2\frac{n-1}{n+1}}\lVert\nabla_\theta \phi\rVert_{L^2(\partial B_1,a)}^{\frac{4}{n+1}}=C_3\Theta^\beta\lVert\nabla_\theta \phi\rVert_{L^2(\partial B_1,a)}^{2(1-\beta)},$$ which is precisely \eqref{log2}.
			
		\end{proof}
	\end{lemma}
	\begin{proof}[Proof of Theorem \ref{thm2}]
		Notice that we can suppose $W_{2m}(z)>0$, otherwise we have done.
		
		First, as already observed, we must show the estimate \eqref{stimaaaa}. %i.e. \begin{equation*}\label{stimaaaa}
		%	\lVert\nabla_\theta \phi\rVert_{L^2(\partial B_1)}^{2}\le C_{n,m}\Theta,
		%\end{equation*} for some $C_{n,m}>0$, that depends only on $n$ and $m$. 
		In fact, as in \eqref{ult} with $\alpha=2m$, we have
		$$\lVert\nabla_\theta \phi\rVert_{L^2(\partial B_1,a)}^{2}\le C_{n,m,a}\sum_{\{k:\ \alpha_k>2m\}} (\lambda_k-\lambda(2m))c_k^2=C_{n,m,a}W_{2m}(r^{2m}\phi), $$ where in the last equality we used \eqref{prima}. 
		
		Using the orthogonality of $P$ and $\phi$ and again \eqref{prima}, we obtain $$\begin{aligned}W_{2m}(r^{2m}\phi)&=W_{2m}(z)-W_{2m}(r^{2m}P(\theta))\\&\le W_{2m}(z)+\frac{\lambda^a(2m)}{n+a+4m-1}\norm P_{L^2(\partial B_1,a)}^2\\&\le W_{2m}(z)+2m\norm P_{L^2(\partial B_1,a)}^2\le W_{2m}(z)+2m\norm c_{L^2(\partial B_1,a)}^2\\&\le(1+2m)\Theta,\end{aligned}$$ where we used that $\norm c_{L^2(\partial B_1)}^2\le\Theta$ and $\lvert W_{2m}(z)\rvert\le\Theta$, that concludes the estimate.
		
		If we choose $\varepsilon <\frac{C_2}{C_1C_3}$, where $C_1,C_2,C_3$ are the constant in \eqref{log1} and \eqref{log2} and if we choose $\kappa_{\alpha,2m}$ as in \eqref{kkkk}, we deduce that \begin{equation}\label{epilog}\begin{aligned}W_{2m}(\zeta)-(1&-\kappa_{\alpha,2m})W_{2m}(z)\le C_1M^2 \kappa_{\alpha,2m}^2-C_2\kappa_{\alpha,2m} \lVert\nabla_\theta \phi\rVert_{L^2(\partial B_1,a)}^{2}\\&\le C_1C_3 \kappa_{\alpha,2m}^2\Theta^\beta\lVert\nabla_\theta \phi\rVert_{L^2(\partial B_1,a)}^{2(1-\beta)}-C_2\kappa_{\alpha,2m} \lVert\nabla_\theta \phi\rVert_{L^2(\partial B_1,a)}^{2}\\&=C_1C_3 \varepsilon ^2\Theta^{-\beta}\lVert\nabla_\theta \phi\rVert_{L^2(\partial B_1,a)}^{4\beta}\lVert\nabla_\theta \phi\rVert_{L^2(\partial B_1,a)}^{2(1-\beta)}\\&\qquad -C_2\varepsilon \Theta^{-\beta}\lVert\nabla_\theta \phi\rVert_{L^2(\partial B_1,a)}^{2\beta} \lVert\nabla_\theta \phi\rVert_{L^2(\partial B_1,a)}^{2}\\&=\varepsilon (\varepsilon C_1C_3-C_2)\Theta^{-\beta}\lVert\nabla_\theta \phi\rVert_{L^2(\partial B_1,a)}^{2+2\beta}<0.
			\end{aligned}
		\end{equation} Finally \begin{equation*}\begin{aligned}W_{2m}(z)&=\frac{1}{n+a+4m-1}\sum_{k=1}^\infty(\lambda^a_k-\lambda^a(2m))c_k^2\\&\le\frac{1}{n+a+4m-1}\sum_{\{k:\ \alpha_k>2m\}}(\lambda_k^a-\lambda^a(2m))c_k^2\\&\le\sum_{\{k:\ \alpha_k>2m\}}\lambda^a_kc_k^2=\lVert\nabla_\theta \phi\rVert_{L^2(\partial B_1,a)}^{2},\end{aligned}
		\end{equation*} therefore $$\begin{aligned}W_{2m}(\zeta)&\le(1-\kappa_{\alpha,2m})W_{2m}(z)=(1-\varepsilon\Theta^{-\beta} \lVert\nabla_\theta \phi\rVert_{L^2(\partial B_1,a)}^{2\beta})W_{2m}(z)\\&\le(1-\varepsilon\Theta^{-\beta} W_{2m}(z)^\beta)W_{2m}(z),\end{aligned} $$ that conclude the proof.
	\end{proof}
	\begin{remark}
		We have proved a stronger version of the logarithmic epiperimetric inequality, that is \begin{equation}\label{stronglog} W_{2m}(\zeta)\le W_{2m}(z)(1-\varepsilon\Theta^{-\beta}\lvert W_{2m}(z)\rvert^\beta)-\varepsilon_1\Theta^{-\beta}\lVert\nabla_\theta \phi\rVert_{L^2(\partial B_1,a)}^{2+2\beta},
		\end{equation} for $\varepsilon _1=\frac{C_2\varepsilon }{2}$, choosing $\varepsilon  <\frac{C_2}{2C_1C_3}$ in \eqref{epilog}.
	\end{remark}
	\section{Epiperimetric inequalities for negative energies}\label{epin}
	In this section we prove two epiperimetric inequalities for negative energies\footnote{Only relevant if $W(z)<0$} $\Ws$ and $\Wm$. 
	
	These epiperimetric inequalities are a generalization for the case $s=\frac12$ and they allow us to prove the backward frequency gap in Proposition \ref{gap}.
	\subsection{Epiperimetric inequality for negative energies $\Ws$}
	In the case $s=\frac12$, the epiperimetric inequality for negative energies $\Ws$ was proved in \cite{car23}. We follow the same idea.
	\begin{theorem}[Epiperimetric inequality for negative energies $\Ws$]\label{thm4} Let $\mathcal{K}_c$ be defined as in \eqref{g} and $z=r^{1+s}c(\theta)\in\mathcal{K}_c$ be the $(1+s)-$homogeneous extension in $\mathbb{R}^{n+1}$ of a function $c\in H^1(\partial B_1,a)$. Then, there is $\zeta\in\mathcal{K}_c$ such that $$\Ws(\zeta)\le (1+\varepsilon)\Ws(z),$$ with $\varepsilon=\frac{1+a}{2n-a+3}.$
	\end{theorem}
	
	\begin{proof} Let $z$ be the $(1+s)-$homogeneous extension of its trace $c\in H^1(\partial B_1,a)$. Then, we can decompose $z$ as $$z(r,\theta)=Cr^{1+s}h^s_e(\theta)+c_0r^{1+s}u_0(\theta)+r^{1+s}\phi(\theta), $$ with $u_0(x,y)=\lvert y\rvert^{2s}$, as in the proof of Theorem \ref{thm1}. Therefore the explicit competitor is $$\zeta(r,\theta)=Cr^{1+s}h^s_e(\theta)+c_0r^{2s}u_0(\theta)+r^{1+s}\phi(\theta),$$
		which is an admissible function, since $\zeta=z\ge0$ on $B'_1$, i.e. $\zeta\in\mathcal{K}_c$.
		
		Now we want to compute the Weiss' energy of $Cr^{1+s}h^s_e(\theta)+c_0r^\alpha u_0(\theta)+r^{1+s}\phi(\theta)$, for $\alpha=2s,{1+s}$. 
		By \eqref{w4}, we have $W_{1+s}(h^s_e)=0$. Then
		\begingroup
		\allowdisplaybreaks
		\begin{align*}
			&W_{1+s}(Cr^{1+s}h^s_e+c_0r^\alpha u_0+r^{1+s}\phi)=C^2W_{1+s}(h^s_e)+W_{1+s}(c_0r^\alpha u_0+r^{1+s}\phi)\\&\qquad+2C\Biggl(\int_{B_1}\nabla (r^{1+s}h^s_e)\cdot \nabla(c_0r^\alpha u_0+r^{1+s}\phi)\y\,dX\\&\qquad-({1+s})\int_{\partial B_1}h^s_e(c_0u_0+\phi)\y\,d\mathcal{H}^n\Biggl) \\&=c_0^2W_{1+s}(r^\alpha u_0)+W_{1+s}(r^{1+s}\phi)\\&\qquad+2c_0\Biggl( \int_{B_1}\nabla (r^\alpha u_0)\cdot \nabla (r^{1+s}\phi)\y\,dX-({1+s})\int_{\partial B_1}u_0\phi\y\,d\mathcal{H}^n\Biggl)\\&\qquad+2C\Biggl( \int_{B_1}\nabla (r^{1+s}h^s_e)\cdot \nabla (r^{1+s}\phi)\y\,dX-({1+s})\int_{\partial B_1}h^s_e\phi\y\,d\mathcal{H}^n\Biggl),
		\end{align*}%
		\endgroup where in the last equality we used that $u_0\equiv$ 0 on $B'_1$, combined with Proposition \eqref{heprop}.
		
		Integrating by parts, we get that  $$\begin{aligned}
			&W_{1+s}(Cr^{1+s}h^s_e+c_0r^\alpha u_0+r^{1+s}\phi)=c_0^2W_{1+s}(r^\alpha u_0)+W_{1+s}(r^{1+s}\phi)\\&\qquad+2c_0\Biggl( \int_{B_1}\nabla (r^\alpha u_0)\cdot \nabla (r^{1+s}\phi)\y\,dX-({1+s})\int_{\partial B_1}u_0\phi\y\,d\mathcal{H}^n\Biggl)\\&\qquad+2C\Biggl(\int_{B_1} -L_a (r^{1+s}h_e^s) r^{1+s}\phi\y \,dX\Biggl).
		\end{aligned}$$
		Then
		
		\begin{equation}\label{ijkl}W_{1+s}(\zeta)- (1+\varepsilon)W_{1+s}(z)=I+J+K+L,
		\end{equation} where $$I=c_0^2\left(W_{1+s}(r^{2s} u_0)-(1+\varepsilon)W_{1+s}(r^{1+s} u_0)\right),$$ $$J=W_{1+s}(r^{1+s}\phi)-(1+\varepsilon)W_{1+s}(r^{1+s}\phi),$$ $$\begin{aligned}K&=2c_0\Biggl( \int_{B_1}\nabla (r u_0)\cdot \nabla (r^{1+s}\phi)\y\,dX-({1+s})\int_{\partial B_1}u_0\phi\y\,d\mathcal{H}^n\\&-(1+\varepsilon)\Biggl(\int_{B_1}\nabla (r^{1+s} u_0)\cdot \nabla (r^{1+s}\phi)\y\,dX-({1+s})\int_{\partial B_1}u_0\phi\y\,d\mathcal{H}^n\Biggl)\Biggl)
		\end{aligned}$$ and $$\begin{aligned}L&=2C\Biggl(\int_{B_1} -L_a (r^{1+s}h_e^s) r^{1+s}\phi \y\,dX\\&\qquad-(1+\varepsilon)\Biggl(\int_{B_1} -L_a (r^{1+s}h_e^s) r^{1+s}\phi \y\,dX\Biggl)\Biggl).
		\end{aligned}$$
		
		For $I$, we notice that the function $-r^{2s}u_0(\theta)=-\lvert y \rvert^{2s} $ is a solution of \eqref{fract4} (with 0 obstacle), then using \eqref{terza}, we obtain $$W_{1+s}(r^{2s}u_0)=W_{1+s}(-r^{{1+s}-(1-s)}u_0)=-(1-s)\lVert u_0\rVert_{L^2(\partial B_1,a)}^2$$ and using \eqref{quarta}, we get \bea W_{1+s}(r^{1+s}u_0)&=\left(1+\frac{-(1-s)}{n+2}\right)W_{1+s}(r^{2s}u_0)\\&=\left(1+\frac{-(1-s)}{n+2}\right)\left(-(1-s)\right)\lVert u_0\rVert_{L^2(\partial B_1,a)}^2,\eea therefore $$I=\left(-(1-s)-\left(1+\varepsilon\right)\left(1+\frac{-(1-s)}{n+2}\right)\left(-(1-s)\right)\right)\lVert u_0\rVert_{L^2(\partial B_1,a)}^2=0,$$ since $\varepsilon =\frac{1+a}{2n-a+3}$, with a simple calculation.
		
		For $J$, using \eqref{prima}, we deduce that $$J=-\varepsilon W_{1+s}(r^{1+s}\phi)=-\frac\varepsilon{n+2}\sum_{\{k:\ \lambda^a_k\ge \lambda^a(2)\}}\left(\lambda^a_k-\lambda^a\left({1+s}\right)\right)c_k^2\le0,$$ since $\lambda^a(2)\ge\lambda^a({1+s})$.
		
		For $K$, we integrate by parts  
		\begingroup
		\allowdisplaybreaks
		\begin{align*}&K=2c_0\Biggl(\int_{B_1}\nabla (r^{2s} u_0)\cdot \nabla (r^{1+s}\phi)\y\,dX-2s\int_{\partial B_1}u_0\phi\y\,d\mathcal{H}^n\\&\qquad-(1-s)\int_{\partial B_1}u_0\phi\y\,d\mathcal{H}^n-(1+\varepsilon)\Biggl(\int_{B_1}\nabla (r^{1+s} u_0)\cdot \nabla (r^{1+s}\phi)\y\,dX\\&\qquad-({1+s})\int_{\partial B_1}u_0\phi\y\,d\mathcal{H}^n\Biggl)\Biggl)
			\\&=2c_0\Biggl( \int_{B_1}-L_a(r^{2s} u_0)r^{1+s}\phi\y\,dX-(1-s)\int_{\partial B_1}u_0\phi\,d\mathcal{H}^n\\&\qquad-(1+\varepsilon)\Biggl(\int_{B_1}-L_a (r^{1+s} u_0)r^{1+s}\phi\y\,dX\Biggl)\Biggl).\end{align*}
		\endgroup
		
		Now, by \eqref{sphh}, we get $$L_a (r^\alpha u_0)\y=\lambda^a(\alpha)r^{\alpha-2}u_0\y+r^{\alpha-2}L_a^S u_0\y,$$ with $\lambda^a(2s)=2ns$ and $\lambda^a({1+s})=(1+s)n+(1+s)(1-s)$, then
		\begingroup
		\allowdisplaybreaks
		\begin{align*} K&=2c_0\Biggl( \int_{B_1}-2nsr^{-2+2s}u_0r^{1+s}\phi\y\,dX-\int_{B_1}r^{-2+2s}L_a^S u_0r^{1+s}\phi\y\,dX\\&\qquad-(1-s)\int_{\partial B_1}u_0\phi\y\,d\mathcal{H}^n\\&\qquad-(1+\varepsilon)\Biggl(\int_{B_1}-\left(({1+s})n+(1+s)(1-s)\right)r^{-1+s}u_0r^{1+s}\phi\,dX\\&\qquad-\int_{B_1}r^{-1+s}L_a^S u_0r^{1+s}\phi\y\,dX\Biggl)\Biggl)\\&=2c_0\Biggl( -\frac{2ns}{n+{3s+a}}\int_{\partial B_1}u_0\phi\y\,d\mathcal{H}^n-\frac{1}{n+{3s+a}}\int_{\partial B_1}L_a^S u_0\phi\y\,d\mathcal{H}^n\\&\qquad-(1-s)\int_{\partial B_1}u_0\phi\y\,d\mathcal{H}^n\\&\qquad-(1+\varepsilon)\Biggl(-\frac{({1+s})n+(1+s)(1-s)}{n+2}\int_{\partial B_1}u_0\phi\,d\mathcal{H}^n\\&\qquad-\frac{1}{n+2}\int_{\partial B_1}L_a^S u_0\phi\y\,d\mathcal{H}^n\Biggl)\Biggl),
		\end{align*}
		\endgroup 
		where in the last equality we used \eqref{w0}.
		Hence, we obtain
		$$\begin{aligned}K&=2c_0\Biggl( -\frac{2ns}{n+{3s+a}}-(1-s)\\&+\left(1+\varepsilon\right)\frac{({1+s})n+(1+s)(1-s)}{n+2}\Biggl) \int_{\partial B_1}u_0\phi\y\,d\mathcal{H}^n \\&+2c_0\Biggl(-\frac{1}{n+{3s+a}}+\left(1+\varepsilon\right)\left(\frac{1}{n+2}\right)\Biggl)\int_{\partial B_1}L_a^S u_0\phi\y\,d\mathcal{H}^n=0,
		\end{aligned}$$ since $\varepsilon =\frac{1+a}{2n-a+3}$, with a simple calculation. 
		
		For $L$, by Proposition \eqref{heprop}, we have \bea L&=-2C\varepsilon\left( \int_{B_1} -L_a(r^{1+s}h_e^s)r^{1+s} \phi\y\,dX\right)\\&=-4c_sC\varepsilon\int_{B'_1\cap\{x\cdot e<0\}}r^{1+s}\phi \lvert x\cdot e\rvert^{1-s}\y\,d\mathcal{H}^n\le0,\eea where we used that $h_e^s(\theta)=u_0(\theta)=0$ for $\theta\in B'_1\cap\{x\cdot e<0\}$ and that $\phi=c\ge0$ in $B'_1\cap\{x\cdot e<0\}$.
		
		Finally, since $I,J,K,L\le0$, we conclude by using \eqref{ijkl}.	
	\end{proof}
	\subsection{Epiperimetric inequality for negative energies $\Wm$}
	In the case $s=\frac12$, the epiperimetric inequality for negative energies $\Wm$ was proved in \cite{csv17}. We follow the same idea.
	\begin{theorem}[Epiperimetric inequality for negative energies $\Wm$]\label{thm5} Let $\mathcal{K}_c$ be defined as in \eqref{g} and let $z=r^{2m}c(\theta)\in\mathcal{K}_c$ be the $2m-$homogeneous extension in $\mathbb{R}^{n+1}$ of a function $c\in H^1(\partial B_1,a)$ such that $\norm c_{L^2(\partial B_1,a)}^2\le1$. Then, there is $\zeta\in\mathcal{K}_c$ such that $$\Wm(\zeta)\le (1+\varepsilon)\Wm(z),$$ with $\varepsilon=\eps(n,m,a)>0$ small enough.
	\end{theorem}
	\begin{proof}Notice that we can suppose $W_{2m}(z)<0$, otherwise we have done.
		
		We can decompose $c\in H^1(\partial B_1,a)$ as $$c=\sum_{k=1}^\infty c_k\phi_k(\theta)=\overline P(\theta)+\overline \phi(\theta),$$ where $\phi_k(\theta)$ are as above, with $$\overline P(\theta)=\sum_{\{k:\ \alpha_k<2m\}}c_k \phi_k(\theta)$$ and $$\overline \phi(\theta)=\sum_{\{k:\ \alpha_k\ge2m\}}c_k \phi_k(\theta).$$ 
		
		We decompose $z=r^{2m}c$, the $2m-$homogeneous extension of $c$, as $$z(r,\theta)=r^{2m}\overline P(\theta)+Mr^{2m}h_{2m}(\theta)-Mr^{2m}h_{2m}(\theta)+r^{2m}\overline \phi(\theta),$$ where $h_{2m}$ as in Remark \ref{h2m} and $M=\max\{P_-(\theta): \theta\in\partial B_1'\}$ as above.
		
		We define $\varepsilon=\kappa^a_{2m,\alpha}$ with $\kappa^a_{2m,\alpha}$ as in $\eqref{kappa}$ and $\alpha\in(2m-\frac12,2m)$ and we will choose $\varepsilon =\varepsilon (n,m,a)>0$ small enough (which corresponds to choosing $\alpha$ close to $2m$ and $\alpha<2m$).
		
		The explicit competitor is $$\zeta(r,\theta)=r^{\alpha}\overline P(\theta)+Mr^{\alpha}h_{2m}(\theta)-Mr^{2m}h_{2m}(\theta)+r^{2m}\overline\phi(\theta),$$
		where we notice that $\zeta\in\mathcal{K}_c$ since $\zeta(r,\theta)\ge r^{2m} c(\theta)\ge0$ on $B'_1$.
		
		As above, we can define
		$$\psi_\mu(r,\theta)=\sum_{\{k:\ \alpha_k<2m\}} c_k r^{\mu} \phi_k(\theta)+Mr^{\mu}h_{2m}(\theta),$$ $$ H_{2m}(r,\theta)=-Mr^{2m} h_{2m}(\theta)+\sum_{\{k:\ \alpha_k=2m\}} c_k r^{2m} \phi_k(\theta)$$ and $$\varphi(r,\theta)=\sum_{\{k:\ \alpha_k>2m\}} c_k r^{2m} \phi_k(\theta), $$
		and we write $z=\psi_{2m}+H_{2m}+\varphi$ and $\zeta=\psi_\alpha+H_{2m}+\varphi$. 
		
		Moreover, since the functions $\psi_\mu$ are orthogonal to $\varphi$ in $L^2( B_1,a)$ and in $H^1( B_1,a)$, since $H_{2m}$ is $L_a-$harmonic and $2m-$homogeneous, and since $W_{2m}(\varphi)>0$ by \eqref{prima}, we have that  
		$$\begin{aligned}
			W_{2m}(\zeta)-(1+&\varepsilon)W_{2m}(z)=W_{2m}(\zeta)-(1+\kappa^a_{2m,\alpha})W_{2m}(z) \\&=-\kappa^a_{2m,\alpha}W_{2m}(\varphi)+W_{2m}(\psi_\alpha)-(1+\kappa^a_{2m,\alpha})W_{2m}(\psi_{2m})\\&\le W_{2m}(\psi_\alpha)-(1+\kappa^a_{2m,\alpha})W_{2m}(\psi_{2m})\\&=M^2\lVert h_{2m}\rVert_{L^2(\partial B_1,a)}^2 \frac{\kappa^a_{2m,\alpha}}{n+a+2\alpha-1}(\lambda^a(2m)-\lambda^a(\alpha))\\&\qquad-\frac{\kappa^a_{2m,\alpha}}{n+a+2\alpha-1}\sum_{\{k:\ \alpha_k<2m\}}(\lambda^a(\alpha)-\lambda^a_k)c_k^2\\&=M^2\lVert h_{2m}\rVert_{L^2(\partial B_1,a)}^2(\kappa^a_{\alpha,2m})^2\frac{(\alpha+2m+n+a-1)^2}{n+a+2\alpha-1} \\&\qquad-\frac{\kappa^a_{\alpha,2m}}{n+a+2\alpha-1}\sum_{\{k:\ \alpha_k<2m\}}(\lambda^a(\alpha)-\lambda^a_k)c_k^2\end{aligned}$$ where in the second equality we used \eqref{seconda}.
		
		Furthermore, if $C_1=\lambda^a(2m-\frac12)-\lambda^a(2m-1)$, then $$\sum_{\{k:\ \alpha_k<2m\}}(\lambda^a(\alpha)-\lambda^a_k)c_k^2 \ge C_1 \sum_{\{k:\ \alpha_k<2m\}}c_k^2=C_1\lVert\overline P\rVert_{L^2(\partial B_1,a)}^2\ge\frac{C_1}{C_0}M^2$$ where the last inequality follows by $$\begin{aligned}M^2&\le\left(\sum_{\{k:\ \alpha_k<2m\}}\lvert c_k \rvert \lVert\phi_k \rVert_{L^\infty(\partial B_1)}\right)^2\le C_0\sum_{\{k:\ \alpha_k<2m\}} \lvert c_k\rvert ^2 \\&=C_0\lVert\overline P\rVert_{L^2(\partial B_1,a)}^2,
		\end{aligned}$$ where we used the Cauchy-Schwartz inequality and that all norms in a finite dimensional space are equivalent.
		
		We deduce that  $$\begin{aligned}
			W_{2m}&(\zeta)-(1+\varepsilon)W_{2m}(z)\\&=M^2\lVert h_{2m}\rVert_{L^2(\partial B_1,a)}^2\varepsilon^2\frac{(\alpha+2m+n+a-1)^2}{n+a+2\alpha-1}-\frac{C_1}{C_0}\frac{ M^2\varepsilon}{n+a+2\alpha-1}\\&\le\frac{M^2\varepsilon}{n+a+2\alpha-1}\left(\lVert h_{2m}\rVert_{L^2(\partial B_1,a)}^2 (4m+n)^2\varepsilon -\frac{C_1}{C_0} \right)<0 ,\end{aligned}$$ where we have chosen $\varepsilon=\varepsilon(n,m,a)>0$ small enough.
	\end{proof}
	\section{Frequency gap}\label{fgap}
	Once we have proved the epiperimetric inequalities in Theorems \ref{thm1}, \ref{thm2}, \ref{thm4} and \ref{thm5}, the proof of the frequency gap is standard, as done in \cite{csv17}.
	
	First notice that without loss of generality we can consider a $\lambda-$homogeneous solution of \eqref{fract4} (with 0 obstacle). In fact, let $v=u^{x_0}$ be a $\lambda-$homogeneous solution of \eqref{fract5} with obstacle $\vf\in C^{k,\gamma}(\R^n)$, $k\ge2$ and $\gamma\in(0,1)$. If $\lambda< k+\gamma$, then $\Phi^{x_0}(0^+,v)=n+a+2\lambda$ by Remark \ref{for}. Thus
	we can consider $z$, the blow-up of $v$ at $x_0\in\Gamma(u)$, that is a solution of \eqref{fract4} (with 0 obstacle). 
	
	\begin{proof}[Proof of Proposition \ref{gap}]
		\textbf{Step 1.} To prove the frequency gap around $1+s$, it is sufficient to check that if $\lambda=1+s+t$ with $t>0$, then $t\ge 1-s$, while if $t<0$, then $t\le-(1-s)$.
		
		Let $c\in H^1(\partial B_1,a)$ be a trace of a $(1+s+t)-$homogeneous solution, say $r^{{1+s}+t}c(\theta)$.
		\begin{enumerate}
			\item
			If $t>0$, then $W_{1+s}(r^{{1+s}+t}c)=t\norm c^2\ab>0,$ by \eqref{terza}. Thus, using the epiperimetric inequality for $\Ws$ (Theorem \ref{thm1}), we obtain
			$$\begin{aligned}0<W_{1+s}(r^{{1+s}+t}c)&\le \left(1-\frac{1+a}{2n+a+5}\right)W_{1+s}(r^{{1+s}}c)\\&=\left(1-\frac{1+a}{2n+a+5}\right)\left(1+\frac{t}{n+2}\right) W_{1+s}(r^{{1+s}+t}c),
			\end{aligned}$$ where in the last equality we used \eqref{quarta}.
			
			Therefore we deduce that
			$$\left(1-\frac{1+a}{2n+a+5}\right)\left(1+\frac{t}{n+2}\right)\ge1,$$ 
			which implies that $t\ge1-s$.
			\item					
			If $t<0$, then $W_{{1+s}}(r^{{1+s}+t}c)=t\norm c_{L^2(\partial B_1,a)}^2<0 $ by \eqref{terza}. Thus, using the epiperimetric inequality for negative energies $\Ws$, i.e. Theorem \ref{thm4}, we obtain $$\begin{aligned}W_{{1+s}}(r^{{1+s}+t}c)&\le \left(1+\frac{1+a}{2n-a+3}\right)W_{{1+s}}(r^{{1+s}}c)\\&=\left(1+\frac{1+a}{2n-a+3} \right)\left(1+\frac{t}{n+2}\right)W_{{1+s}}(r^{{1+s}+t}c),
			\end{aligned}$$ where we used \eqref{quarta}.
			
			Hence, since we have a negative energy, we get $$\left(1+\frac{1+a}{2n-a+3}\right)\left(1+\frac{t}{n+2}\right)\le 1, $$ which gives $t\le- ({1+s})$.
		\end{enumerate}
		\textbf{Step 2.}	
		Let $c\in H^1(\partial B_1,a)$ be a trace of a $(2m+t)-$homogeneous solution, say $r^{2m+t}c(\theta)$, with $\norm c_{L^2(\partial B_1,a)}^2=1$. It is sufficient to check that if $t>0$, then $t\ge c_m^+$, for some $c_m^+>0$, while if $t<0$, then $t\le- c_m^-$, for some $c_m^->0$.
		
		\begin{enumerate}\item Let $t>0$, then \begin{equation}\label{subito}0<t=W_{2m}(r^{2m+t}c)\le W_{2m}(r^{2m}c),
			\end{equation} by \eqref{terza}.
			
			Using the logarithmic epiperimetric inequality for $\Wm$, i.e. Theorem \ref{thm2}, and \eqref{subito}, we obtain
			$$\begin{aligned}0<W_{2m}(r^{2m+t}c)&\le W_{2m}(\zeta)\le (1-\varepsilon W_{2m}^\beta(r^{2m}c))W_{2m}(r^{2m}c)\\&\le (1-\varepsilon t^\beta)W_{2m}(r^{2m}c)\\&=(1-\varepsilon t^\beta)\left(1+\frac{t}{n+a+4m-1}\right)W_{2m}(r^{2m+t}c),
			\end{aligned}$$ where in the last equality we used \eqref{quarta}. 
			Therefore, since we have a negative energy, we have $$(1-\varepsilon t^\beta)\left(1+\frac{t}{n+a+4m-1}\right)\ge 1, $$ which gives that  $t\ge c_m^+$, for some explicit constant $c_m^+>0$.
			
			\item If $t<0$ then $ W_{2m}(r^{2m+t}c)=t<0 $ by \eqref{terza}, therefore using the epiperimetric inequality for negative energies $\Wm$, i.e. Theorem \ref{thm5}, we obtain $$\begin{aligned}W_{2m}(r^{2m+t}c)&\le (1+\varepsilon)W_{2m}(r^{2m}c)\\&=(1+\varepsilon )\left(1+\frac{t}{n+a+4m-1}\right)W_{2m}(r^{2m+t}c),
			\end{aligned}$$ where in the last equality we used \eqref{quarta}. 
			
			Thus we get $$(1+\varepsilon )\left(1+\frac{t}{n+a+4m-1}\right)\le 1, $$ which is $t\le -c_m^-$ for some explicit constant $c_m^->0.$
		\end{enumerate}
	\end{proof}
	\section{Characterization of blow-ups}\label{char}
	The epiperimetric inequality approach allows us to give an alternative proof of the characterization of blow-ups, in the spirit of \cite{csv17}. For the original proof we refer to \cite{css08} and \cite{gr17}.
	\begin{proof}[Proof of Proposition \ref{cha}]\textbf{Step 1.} We want to prove that if $z$ is a solution of \eqref{fract4} (with 0 obstacle) and it is $(1+s)-$homogeneous, then $z=Ch_e^s$, $e\in\partial B'_1$ and $C\ge0$.
		
		In fact, suppose that $c\in H^1(\partial B_1,a)$ is the trace of $z$ and let $\zeta$ be the competitor in the epiperimetric inequality for $\Ws$. Then $W_{1+s}(z)=0$, by Proposition \ref{heprop}. Therefore $$0=W_{1+s}(z)\le W_{1+s}(\zeta)\le (1-\kappa)W_{1+s}(z)=0,$$ i.e. the epiperimetric inequality is an equality. 
		Thus, by Remark \ref{epi2}, $c(\theta)=Ch_e^s(\theta)+\phi(\theta)$, for some $e\in\partial B'_1$, $C\ge0$ and $\phi$ eigenfunction corresponding to eigenvalue $\lambda^a(2)$ with $\phi\ge 0$ on $B'_1$. 
		
		Hence, using Proposition \ref{heprop} with an integration by parts, we get $$\begin{aligned}0&=W_{1+s}(z)=W_{1+s}(Cr^{1+s}h_e^s+r^{1+s}\phi)=W_{1+s}(r^{1+s}h_e^s)+W_{1+s}(r^{1+s}\phi)\\&\qquad+2\left(\int_{B_1}\nabla h_e^s \cdot \nabla (r^{1+s}\phi) \y\,dX-({1+s}) \int_{\partial B_1} h_e^s\phi\y\,d\mathcal{H}^n\right)\\&\ge W_{1+s}(r^{1+s}h_e^s)+W_{1+s}(r^{1+s}\phi)=W_{1+s}(r^{1+s}\phi)\ge0,\end{aligned}$$ where we used \eqref{prima} and \eqref{w4}. In particular the last inequality is actually an equality. Then, by \eqref{prima}, we get that $\phi\equiv0$, i.e. $z=Ch_e^s$ for some $C\ge0$. 
		
		\textbf{Step 2.}
		Suppose that $z$ is a solution of \eqref{fract4} (with 0 obstacle) and that $z$ is $2m-$homogeneous. Then we claim that $z=p_{2m}$ for some polynomial $p_{2m}$ which is $L_a-$harmonic.
		
		Let $c\in H^1(\partial B_1,a)$ be the trace of $z$ and let $\zeta$ be the competitor in the logarithmic epiperimetric inequality for $\Wm$. Without loss of generality, we can suppose $\lVert c\rVert_{L^2(\partial B_1,a)}=1$. Hence, by the strong version of the log-epiperimetric inequality \ref{stronglog}, we have $$\begin{aligned} 0&=W_{2m}(z)\le W_{2m}(\zeta)\le W_{2m}(z)(1-\varepsilon\lvert W_{2m}(z)\rvert^\beta)-\varepsilon_1\lVert\nabla_\theta \phi\rVert_{L^2(\partial B_1,a)}^{2+2\beta}\\&=-\varepsilon_1\lVert\nabla_\theta \phi\rVert_{L^2(\partial B_1,a)}^{2+2\beta}
		\end{aligned}$$ i.e. $\lVert\nabla_\theta \phi\rVert_{L^2(\partial B_1,a)}=0$. 
		
		Thus $\phi\equiv0$ and $c$ contains only eigenfunctions corresponding to eigenvalues $\lambda^a_k\le\lambda^a(2m)$, i.e. $$c(\theta)=\sum_{\{k:\ \alpha_k\le2m\}}c_k\phi_k.$$
		
		Using \eqref{prima}, we obtain $$0=W_{2m}(z)=\frac{1}{n+4m-1}\sum_{\{k:\ \alpha_k\le2m\}}(\lambda^a_k-\lambda^a(2m))c_k^2\le0,$$ i.e. the frequencies $\alpha_k<2m$ must vanish. Therefore $c$ is an eigenfunction corresponding to eigenvalue $\lambda^a(2m)$ and it follows that the homogeneous extension $z$ is a $2m-$homogeneous $L_a-$harmonic polynomial. 
	\end{proof}
	\section{Regularity of $\Gamma_{1+s}(u)$ and structure of $\Gamma_{2m}(u)$}\label{final}
	We conclude this paper with the most important application of the epiperimetric inequalities in Theorem \ref{thm1} and Theorem \ref{thm2}, i.e. the proof of Theorem \ref{thm3}.
	The proof following a standard argument, for instance see \cite{gpps17} for the case $s\in(0,1)$ or \cite{fs16}, \cite{gps16}, \cite{csv17}, for the case $s=\frac12.$
	
	In this last section, we recall the results from \cite{gpps17}, where the claim (1) was proved in the case $s\in(\frac12,1)$.
	The regularity assumption for $\vf$, that is $\vf\in C^{k,\gamma}(\R^n)$, allows us to generalize the result to any $s\in(0,1)$ by using the same argument. 
	
	Indeed, in \cite{gpps17} it was proved that the function $r\mapsto\W^{x_0}_{\lambda}(r,v)+Cr^{2s-1}$ is increasing (with a slight difference in the definition of $v$), where $2s-1>0$. We notice that in the case $\vf\in C^{k,\gamma}(\R^n)$, we have that the function $r\mapsto\W^{x_0}_\lambda(r,v)+Cr^{k+\gamma-\lambda}$ is increasing, with $k+\gamma-\lambda>0$, by \eqref{monoW}, with $v=u^{x_0}$ the solution of \eqref{fract5}.
	
	For the case (2), we give the complete proof, which is based on similar arguments and uses ideas from \cite{csv17} and \cite{gp09}. Notice that in the case $\lambda=2m$, the condition $\lambda< k+\gamma$ become $\lambda=2m\le k$ since $2m$ and $k$ are integers.
	\subsection{Decay of the Weiss' energy $\W$} The first result, that follows by the epiperimetric inequalities, is the decay of the Weiss' energy $\W$ for obstacle $\vf\not\equiv 0.$ 
	
	\begin{proposition}[Decay of the Weiss' energy $\W$]\label{dec} Let $u$ be a solution of \eqref{fract2} and $v=u^{x_0}$ be the solution of \eqref{fract5} with $x_0\in\Gamma(u)$, $\vf\in C^{k,\gamma}(\R^n)$, $k\ge2$ and $\gamma\in(0,1)$. Let $K\subset\Gamma_\lambda(u)\cap \R^{n+1}$ be a compact set. \begin{enumerate}
			\item If $\lambda=1+s$, then there is $\alpha\in(0,1)$ such that\be\label{wsdecay}\W^{x_0}_{1+s}(r,v)\le Cr^\alpha  \ee for all $x_0\in\Gamma_{1+s}(u)\cap K$ and for all $r\in(0,r_0)$, for some $r_0>0$.
			\item If $\lambda=2m\le k$ and $\beta\in(0,1)$ is the constant in Theorem \ref{thm2}, then\be\label{wmdecay} \W^{x_0}_{2m}(r,v)\le C(- \log(r)^{-\frac1\beta})\ee for all $x_0\in\Gamma_{2m}(u)\cap K$ and for all $r\in(0,r_0)$, for some $r_0>0$.
		\end{enumerate}
		\begin{proof} Using \eqref{h'}, we obtain
			\begingroup
			\allowdisplaybreaks
			\begin{align*} \frac{d}{dr}&\W^{x_0}_\lambda(r,v)=-\frac{n+a+2\lambda-1}{r^{n+a+2\lambda}}\mathcal{I}(r)+\frac{\mathcal{I}'(r)}{r^{n+a+2\lambda-1}}+\lambda\frac{n+a+2\lambda}{r^{n+a+2\lambda+1}}H(r)\\&\qquad-\frac{\lambda}{r^{n+a+2\lambda}}H'(r)\\&=-\frac{n+a+2\lambda-1}{r}\W^{x_0}_\lambda(r,v)+\frac{\lambda}{r^{n+a+2\lambda+1}}H(r)+\frac{\mathcal{I}'(r)}{r^{n+a+2\lambda-1}}\\&\qquad-\frac{\lambda}{r^{n+a+2\lambda}}H'(r)\\&=-\frac{n+a+2\lambda-1}{r}\W^{x_0}_\lambda(r,v)+\frac{1}{r^{n+a+2\lambda-1}}\int_{\partial B_r(x_0)} \lvert \nabla v\rvert^2\y\,d\HH^n\\&\qquad+\frac{1}{r^{n+a+2\lambda-1}}\int_{\partial B_r(x_0)} vh\,d\HH^n-\frac{2\lambda}{r^{n+a+2\lambda}}\int_{\partial B_r(x_0)} v\partial_{\nu}v\y\,d\HH^n\\&\qquad-\frac{\lambda(n+a-1)}{r^{n+a+2\lambda+1}}H(r)\\&\ge-\frac{n+a+2\lambda-1}{r}\W^{x_0}_\lambda(r,v)+\frac{1}{r^{n+a+2\lambda-1}}\int_{\partial B_r(x_0)} (\partial_{\nu}v)^2\y\,d\HH^n\\&\qquad+\frac{1}{r^{n+a+2\lambda-1}}\int_{\partial B_r(x_0)}\lvert \nabla_\theta v \rvert^2\y\,d\HH^n-Cr^{k+\gamma-\lambda-1}\\&\qquad-\frac{2\lambda}{r^{n+a+2\lambda}}\int_{\partial B_r(x_0)} v\partial_{\nu}v\y\,d\HH^n-\frac{\lambda(n+a-1)}{r^{n+a+2\lambda+1}}H(r),
			\end{align*}
			\endgroup
			where in the last inequality we used \eqref{est}.
			
			By the definition of the homogeneous rescalings $v_r$ from \eqref{hom}, we have
			\begingroup
			\allowdisplaybreaks
			\begin{align*}
				\frac{d}{dr}&\W^{x_0}_\lambda(r,v)\ge-\frac{n+a+2\lambda-1}{r}\W^{x_0}_\lambda(r,v)+\frac{1}{r}\int_{\partial B_1}(\partial_{\nu}v_r-\lambda v_r)^2\,d\HH^n\\&\qquad+\frac{1}{r}\int_{\partial B_1}\lvert \nabla_\theta v_r\rvert^2\,d\HH^n-Cr^{k+\gamma-\lambda-1}\\&\qquad-\frac{\lambda(n+a+\lambda-1)}{r}\int_{\partial B_1} v_r^2\y\,d\HH^n\\&=-\frac{n+a+2\lambda-1}{r}\W^{x_0}_\lambda(r,v)+\frac{1}{r}\int_{\partial B_1}(\partial_{\nu}v_r-\lambda v_r)^2\,d\HH^n\\&\qquad+\frac{1}{r}\int_{\partial B_1}\lvert \nabla_\theta v_r\rvert^2\,d\HH^n-Cr^{k+\gamma-\lambda-1}\\&\qquad-\frac{\lambda(n+a+2\lambda-1)}{r}\int_{\partial B_1} v_r^2\y\,d\HH^n+\frac{\lambda^2}{r}\int_{\partial B_1} v_r^2\y\,d\HH^n\\&=-\frac{n+a+2\lambda-1}{r}\W^{x_0}_\lambda(r,v)+\frac{1}{r}\int_{\partial B_1}(\partial_{\nu}v_r-\lambda v_r)^2\,d\HH^n\\&\qquad-Cr^{k+\gamma-\lambda-1}-\frac{\lambda(n+a+2\lambda-1)}{r}\int_{\partial B_1} v_r^2\y\,d\HH^n\\&\qquad+\frac{n+a+2\lambda-1}{r}\int_{B_1}\lvert \nabla c_r\rvert^2\y\,dX,
			\end{align*}
			\endgroup
			where $c_r$ is the $\lambda-$homogeneous extension of $v_r|_{\partial B_1}$ and where in the last equality we used \eqref{w1}.
			
			Hence we have the following inequality
			\be\label{123}
			\begin{aligned} \frac{d}{dr}\W^{x_0}_\lambda(r,v)&\ge\frac{n+a+2\lambda-1}{r}\left(W_\lambda(1,c_r)-\W^{x_0}_\lambda(r,v)\right)\\&\qquad+\frac{1}{r}\int_{\partial B_1}(\partial_{\nu}v_r-\lambda v_r)^2\,d\HH^n-Cr^{k+\gamma-\lambda-1}\\&\ge\frac{n+a+2\lambda-1}{r}\left(W_\lambda(1,c_r)-\W^{x_0}_\lambda(r,v)\right)-Cr^{k+\gamma-\lambda-1}.
			\end{aligned}
			\ee
			Fixed $r>0$, if $\zeta_r$ is a minimizer of $\mathcal{E}(w)=\int_{B_1}\lvert \nabla w\rvert^2\y\,dX$ among the functions $\{w\in H^1(B_1,a): w\ge0\mbox{ on }B'_1, w=v_r=c_r\mbox{ in }\partial B_1\}$ and $$\zeta(x,y)=r^\lambda \zeta_r\left(\frac {x-x_0}{r},\frac yr\right)$$ is the function such that the homogeneous rescaling of $\zeta$ at $x_0$ is $\zeta_r$, then $v$ is the minimizer of $$\int_{B_r(x_0)}\lvert \nabla w\rvert^2\y\,dX+2\int_{B_r(x_0)}wh\,dX$$ among the functions $\{w\in H^1(B_r(x_0),a): w\ge0\mbox{ on }B'_r(x_0), w=\zeta\mbox{ in }\partial B_r(x_0)\}$.
			Thus \be\label{11} W_\lambda(1,\zeta_r)= W^{x_0}_\lambda(r,\zeta)\ge \W^{x_0}_\lambda(r,\zeta)-Cr^{k+\gamma-\lambda}\ge\W^{x_0}_\lambda(r,v)-Cr^{k+\gamma-\lambda}, \ee 
			since $H(r,\zeta)=H(r,v)$, then we can use \eqref{h1} to deduce the estimate \eqref{est} for $\int_{B_r}\zeta h\,dX$.
			
			Now we separate the case $\lambda=1+s$ and $\lambda=2m$.
			
			\textbf{Step 1.} In the case $\lambda=1+s$, we have that $$W_{1+s}(1,c_r)\ge \frac {1}{1-\kappa}W_{1+s}(1,\zeta_r)$$ by epiperimetric inequality for $\Ws$ (Theorem \ref{thm1}). Therefore $$\frac{d}{dr}\W^{x_0}_{1+s}(r,v)\ge\frac{c}{r}\W^{x_0}_{1+s}(r,v)-Cr^{k+\gamma-\lambda-1},$$ where we used \eqref{123} and \eqref{11}.
			
			Now for all $\alpha\in(0,1)$, we have \bea\frac{d}{dr}\left(\W^{x_0}_{1+s}(r,v)r^{-\alpha}\right)&=\frac{d}{dr}\W^{x_0}_{1+s}(r,v)r^{-\alpha}-\alpha\W^{x_0}_{1+s}(r,v)r^{-\alpha-1}\\&\ge (c-\alpha)\W^{x_0}_{1+s}(r,v)r^{-\alpha-1}-Cr^{k+\gamma-\lambda-1-\alpha}\\&\ge -C(c-\alpha)r^{k+\gamma-\lambda-1-\alpha}-Cr^{k+\gamma-\lambda-1-\alpha}\\&=-Cr^{k+\gamma-\lambda-1-\alpha}
			\eea where we have chosen $\alpha<c$ and we used \eqref{ric}.
			
			Integrating from $r$ to $r_0$, we obtain \bea\W^{x_0}_{1+s}(r_0,v)r_0^{-\alpha}-\W^{x_0}_{1+s}(r,v)r^{-\alpha}&\ge -Cr_0^{k+\gamma-\lambda-\alpha}+Cr^{k+\gamma-\lambda-\alpha}\\&\ge-Cr_0^{k+\gamma-\lambda-\alpha},\eea where we used that $\alpha<k+\gamma-\lambda$. Thus, \eqref{wsdecay} follows.
			
			\textbf{Step 2.} In the case $\lambda=2m$, since $W_{2m}(1,c_r)\ge W_{2m}(1,\zeta_r)\ge0$, we have that \bea W_{2m}(1,c_r)&\ge  W_{2m}(1,\zeta_r)+\eps\Theta^{-\beta} W_{2m}(1,c_r)^{1+\beta}\\&\ge W_{2m}(1,\zeta_r)+\eps\Theta^{-\beta} W_{2m}(1,\zeta_r)^{1+\beta}\eea by epiperimetric inequality (Theorem \ref{thm2}) for $\Wm$ (see Remark \ref{blo} below). Then \be\label{t}\frac{d}{dr}\W^{x_0}_{2m}(r,v)\ge\frac{c}{r} \W^{x_0}_{2m}(r,v)^{1+\beta}-Cr^{k+\gamma-\lambda-1},\ee where we used \eqref{123}, \eqref{11} and where we suppose $\W^{x_0}_{2m}(r,v)\ge0$, since otherwise we were done.

			%%%%%%%%%%%%%%%
			
			Let $F(r)=\max\{\W^{x_0}_{2m}(r,v),\left(\frac2c\right)^{\frac1{1+\beta}}Cr^{\frac{k+\gamma-\lambda}{1+\beta}}\}$, then we claim that $$F'(r)\ge\frac c{2r}F(r)^{1+\beta}, $$ for all $r\in(0,r_0)$, up to decreasing $r_0$. In fact, if $F'(r)=\frac{d}{dr}\W^{x_0}_{2m}(r,v)$, then $$F'(r)\ge\frac c{2r}\W^{x_0}_{2m}(r,v)^{1+\beta}=\frac c{2r}F(r)^{1+\beta},$$by \eqref{t} and 
			since $\W^{x_0}_{2m}(r,v)\ge\left(\frac2c\right)^{\frac1{1+\beta}}Cr^{\frac{k+\gamma-\lambda}{1+\beta}}.$ While, if $$F'(r)= \left(\frac2c\right)^{\frac1{1+\beta}}C\left(\frac{k+\gamma-\lambda}{1+\beta}\right)r^{\frac{k+\gamma-\lambda}{1+\beta}-1},$$ then \bea F'(r)\ge \frac c2\left(\frac2c\right)^{\frac1{1+\beta}}Cr^{k+\gamma-\lambda-1}=\frac c{2r}F(r)^{1+\beta},\eea
			since $\W^{x_0}_{2m}(r,v)\le\left(\frac2c\right)^{\frac1{1+\beta}}Cr^{\frac{k+\gamma-\lambda}{1+\beta}},$ that is the claim
			
			\begin{comment}
			Let $F(r)=\max\{\W^{x_0}_{2m}(r,v),\left(\frac2c\right)^{\frac1{1+\beta}}Cr^{\frac{k+\gamma-\lambda}{1+\beta}}\}$, then we claim that $$F'(r)\ge\frac c{2r}F(r)^{1+\beta}, $$ for all $r\in(0,r_0)$, up to decreasing $r_0$. In fact, if $\W^{x_0}_{2m}(r,v)>\left(\frac2c\right)^{\frac1{1+\beta}}Cr^{\frac{k+\gamma-\lambda}{1+\beta}},$ then $$F'(r)=\frac{d}{dr}\W^{x_0}_{2m}(r,v)\ge\frac c{2r}\W^{x_0}_{2m}(r,v)^{1+\beta}=\frac c{2r}F(r)^{1+\beta},$$by \eqref{t}. While, if $\W^{x_0}_{2m}(r,v)<\left(\frac2c\right)^{\frac1{1+\beta}}Cr^{\frac{k+\gamma-\lambda}{1+\beta}},$ then \bea F'(r)&= \left(\frac2c\right)^{\frac1{1+\beta}}C\left(\frac{k+\gamma-\lambda}{1+\beta}\right)r^{\frac{k+\gamma-\lambda}{1+\beta}-1}\ge \frac c2\left(\frac2c\right)^{\frac1{1+\beta}}Cr^{k+\gamma-\lambda-1}\\&=\frac c{2r}F(r)^{1+\beta},\eea that is the claim
				\end{comment}
			Now notice that the function $r\mapsto-F(r)^{-\beta}-\frac c2\beta\log(r)$ is increasing for $r\in(0,r_0)$. In fact, by the previous inequality, we have\bea \frac{d}{dr}\left(-\frac1\beta F(r)^{-\beta}-\frac c2\log(r)\right)&=F(r)^{-1-\beta}F'(r)- \frac c{2r}\ge0.
			\eea
			
			Therefore, if we choose $r_0>0$ such that $\W_{2m}^{x_0}(r_0,v)\ge0$, we have \bea F(r)^{-\beta}&\ge -\frac c2\beta\log(r)+F(r_0)^{-\beta}+\frac c2\beta\log(r_0)&\ge-\frac c2\beta\log\left(\frac{r}{r_0}\right),
			\eea that is \bea \W_{2m}^{x_0}(r,v)\le F(r)\le C\left(- \log\left(\frac r{r_0}\right)^{-\frac1\beta}\right)
			\eea
			Finally, if $r\le r_0^2$, then $\log\left(\frac r{r_0}\right)\le\frac{1}{2}\log(r)$ and \eqref{wmdecay} follows.

			%%%%%%%%%%%%%%%%%%%
			\begin{comment}
		
			Moreover, we also suppose that $\W^{x_0}_{2m}(r,v)>\left(\frac2c\right)^{\frac1{1+\beta}}Cr^{\frac{k+\gamma-\lambda}{1+\beta}},$ since otherwise we were done. Since $\W^{x_0}_{2m}(r,v)^{1+\beta}\ge\frac2c Cr^{k+\gamma-\lambda},$ then $$\frac{d}{dr}\W^{x_0}_{2m}(r,v)\ge\frac c{2r}\W^{x_0}_{2m}(r,v)^{1+\beta},$$by \eqref{t}.
			
			Now notice that the function $r\mapsto-\W_{2m}^{x_0}(r,v)^{-\beta}-\frac c2\beta\log(r)$ is increasing for $r\in(0,r_0)$. In fact, by the previous inequality, we have\bea \frac{d}{dr}\left(-\frac1\beta\W^{x_0}_{2m}(r,v)^{-\beta}-\frac c2\log(r)\right)&=\W^{x_0}_{2m}(r,v)^{-1-\beta}\frac{d}{dr}\W^{x_0}_{2m}(r,v)- \frac c{2r}\ge0.
			\eea
			
			Therefore, if we choose $r_0>0$ such that $\W_{2m}^{x_0}(r_0,v)\ge0$, we have \bea \W_{2m}^{x_0}(r,v)^{-\beta}&\ge -\frac c2\beta\log(r)+\W_{2m}^{x_0}(r_0,v)^{-\beta}+\frac c2\beta\log(r_0)&\ge-\frac c2\beta\log\left(\frac{r}{r_0}\right),
			\eea that is \bea \W_{2m}^{x_0}(r,v)\le C\left(- \log\left(\frac r{r_0}\right)^{-\frac1\beta}\right)
			\eea
			Finally, if $r\le r_0^2$, then $\log\left(\frac r{r_0}\right)\le\frac{1}{2}\log(r)$ and \eqref{wmdecay} follows.
		\end{comment}
			%%%%%%%%%%%%%%%%%%%
			
		\end{proof} 
	\end{proposition}
	\begin{remark}\label{blo}
		In the Step 2 of the proof of Proposition \ref{dec}, we used the logarithmic epiperimetric inequality for the rescaled $c_r$, but to use Theorem \ref{thm2}, we have to check that the conditions \eqref{cond} hold with $\Theta>0$ that does not depend on $r$.
		
		In fact $\lVert c_r\rVert^2_{L^2(\partial B_1,a)}=\lVert v_r\rVert^2_{L^2(\partial B_1,a)}\le \Theta$ as in the proof of Proposition \ref{prop2}.
		
		Moreover, by \eqref{w2} \bea W_{2m}(c_r)&\le C\lVert \nabla_\theta c_r\rVert^2 \ab\le C\lVert \nabla c_r\rVert^2 \aB\\&=C(\lVert \nabla_\theta v_r\rVert^2 \ab+\lambda^2\lVert  v_r\rVert^2 \ab)\\&\le C(\lVert \nabla v_r\rVert^2 \aB+\lambda^2\lVert  v_r\rVert^2 \ab)\le \Theta, \eea where in the equality we used \eqref{w1} and in the last inequality we used the boundness of $v_r$ in $H^1(B_1,a)$, as in the proof of Proposition \eqref{prop2}.
	\end{remark}
	
	\subsection{Decay of homogeneous rescalings}
	The decay of the Weiss' energy allows us to prove a decay of the norm in $L^1(\partial B_1,a)$ of the homogeneous rescalings. As a consequence, we get the uniqueness of the homogeneous blow-up.
	\begin{proposition}[Decay of homogeneous rescalings] \label{homdec}Let $u$ be a solution of \eqref{fract2} and $v=u^{x_0}$ be the solution of \eqref{fract5} with $x_0\in\Gamma(u)$, $\vf\in C^{k,\gamma}(\R^n)$, $k\ge2$ and $\gamma\in(0,1)$. Let $v^{(\lambda)}_{r,x_0} $ be the homogeneous rescalings from \eqref{hom}, and $K\subset\Gamma_\lambda(u)\cap \R^{n+1}$ be a compact set.
		\begin{enumerate}
			\item If $\lambda=1+s$, then, up to decreasing $\alpha$ in Proposition \ref{dec}, we have\be\label{homdecays} \int_{\partial B_1} \lvert v^{(1+s)}_{r,x_0}-v^{(1+s)}_{r',x_0}\rvert\y \,d\HH^n\le Cr^\alpha\ee for all $x_0\in\Gamma_{1+s}(u)\cap K$ and for all $0<r'<r<r_0$, for some $r_0>0$.
			\item If $\lambda=2m\le k$, and $\beta\in(0,1)$ is the constant in Theorem \ref{thm2}, then \be\label{homdecaym} \int_{\partial B_1} \lvert v^{(2m)}_{r,x_0}-v^{(2m)}_{r',x_0}\rvert\y \,d\HH^n\le C(- \log(r)^{-\frac{1}{2\beta}} )\ee for all $x_0\in\Gamma_{2m}(u)\cap K$ and for all $0<r'<r<r_0$, for some $r_0>0$.
		\end{enumerate}
		
		In particular, the homogeneous blow-up $v^{(\lambda)}_{0,x_0} $ is unique and the whole sequence $v^{(\lambda)}_{r,x_0} $ converges to $v^{(\lambda)}_{0,x_0} $ as $r\to0^+$.
		\begin{proof} Dropping the dependence on $x$ and $\lambda$, for $0<r'<r<r_0$, where $r_0$ is as in the previous Proposition, using \eqref{monoW}, we get 
			\begingroup
			\allowdisplaybreaks
			\begin{align*}
				\int_{\partial B_1} \lvert &v_{r}-v_{r'}\rvert\y \,d\mathcal{H}^n=\int_{\partial B_1}\int_{r'}^r\left\lvert \frac{d}{dt}v_{t} \right\rvert \y\,dt\,d\mathcal{H}^n \\&=\int_{\partial B_1}\int_{r'}^r\frac1t\left\lvert \nabla v_t\cdot \nu-\lambda v_t\right\rvert\y\,dt\,d\mathcal{H}^n\\&\le C\int_{r'}^r\frac{1}{t^\frac12} \left(\frac{1}{t}\int_{\partial B_1}\left\lvert\nabla v_t\cdot \nu-\lambda v_t\right\rvert^2 \y\,d\mathcal{H}^n\right)^\frac12\,dt\\&\le C\int_{r'}^r \frac{1}{t^\frac12}\left(\frac{d}{dt}\left(\W^{x_0}_\lambda(t,v)+Cr^{k+\gamma-\lambda}\right)\right)^\frac12\,dt\\&\le C\left(\int_{r'}^r \frac{1}{t}\,dt\right)^\frac12\left(\int_{r'}^r \frac{d}{dt}\left(\W^{x_0}_\lambda(t,v)+Cr^{k+\gamma-\lambda}\right)\,dt\right)^\frac12\\&=C\log\left(\frac r{r'}\right)^\frac12(\W^{x_0}_\lambda(r,v)+Cr^{k+\gamma-\lambda}-\W^{x_0}_\lambda(r',v)-C{(r')}^{k+\gamma-\lambda})^\frac12\\&\le C\log\left(\frac r{r'}\right)^\frac12\left(\W^{x_0}_\lambda(r,v)+r^{k+\gamma-\lambda}\right)^\frac12,
			\end{align*}
			\endgroup	
			where we used that $\W^{x_0}_\lambda(r',v) +C(r')^{k+\gamma-\lambda}\ge0$, by \eqref{monoW} and \eqref{ric}.
			
			The conclusion follows by a dyadic decomposition as in \cite{fs16}, \cite{gps16} or \cite{csv17}, and by using \eqref{wsdecay} for $\lambda=1+s$, and \eqref{wmdecay} for $\lambda=2m$.
		\end{proof}
	\end{proposition}
	\subsection{Non-degeneracy of homogeneous blow-up}
	Another consequence of the decay of the Weiss' energy is the non-degeneracy of the homogeneous blow-ups, i.e. the homogeneous blow-ups cannot vanish identically.
	\begin{proposition}[Non-degeneracy of homogeneous blow-up] \label{nondeg} Let $u$ be a solution of \eqref{fract2} and $v=u^{x_0}$ be the solution of \eqref{fract5} with $x_0\in\Gamma(u)$, $\vf\in C^{k,\gamma}(\R^n)$, $k\ge2$ and $\gamma\in(0,1)$. If $\lambda=1+s$ or $\lambda=2m\le k$, then \bea H^{x_0}(r)\ge H_0^{x_0}r^{n+a+2\lambda} \eea for all $x_0\in\Gamma_\lambda(u)$ and for all $r\in(0,r_0)$, for some $r_0>0$, where $$H_0^{x_0}:=\lim_{r\to0^+}\frac{H^{x_0}(r)}{r^{n+a+2\lambda}}>0.$$ 
		In particular the homogeneous blow-up $v^{(\lambda)}_{0,x} $ is non-trivial, since $$\lVert v^{(\lambda)}_{0,x_0} \rVert_{L^2(\partial B_1,a)}=H_0^{x_0}>0.$$
		\begin{proof}
			Note that the inequality follows by the proof of \eqref{h1}, in fact we can obtain that the function $r\mapsto \frac{H^{x_0}(r)}{r^{n+a+2\lambda}}$ is increasing for $r$ small enough.
			Hence it is sufficient to prove that $H_0^{x_0}>0$.

			Let $v_r$ be the homogeneous rescalings as in \eqref{hom} of $v$ at $x_0$, then $v_r$ converge in $C^{1,\alpha}_a$ to some $v_0$, $\lambda-$homogeneous solution, as $r\to0^+$, up to subsequences.
			
			Let $(v_r)_\rho$ be the rescalings as in \eqref{rescaling} of $v_r$ at 0, then $(v_r)_\rho$ converge in $C^{1,\alpha}_a$ to some $\widetilde v_r$, $\lambda-$homogeneous solution, as $\rho\to0^+$, up to subsequences. 
			
			Arguing by contradiction, we suppose that $H^{x_0}_0=0$. Then \bea\lVert \widetilde v_r\rVert^2\ab&=\lim_{\rho\to0^+}H(1,(v_r)_\rho)=\lim_{\rho\to0^+}\frac{H(\rho,v_r)}{\rho^{n+a+2\lambda}}\\&=\frac{r^{n+a+2\lambda}}{H^{x_0}(r,v)}\lim_{\rho\to0^+}\frac{H^{x_0}(\rho r,v)}{(\rho r)^{n+a+2\lambda}}=0.\eea
			Therefore, using \eqref{homdecays} and \eqref{homdecaym}, together with the regularity of the solution, we obtain $$\lVert (v_r)_\rho \rVert^2\ab\le C\omega(\rho),$$ with $\omega(\rho)\to0$ as $\rho\to0^+.$
			
			Since $v_0$ is $\lambda-$homogeneous, we have \bea1&=H(1,v_0)=\frac{1}{\rho^{n+a+2\lambda}}H(\rho,v_0)\\&\le\frac{1}{\rho^{n+a+2\lambda}}\lVert v_0-v_r \rVert^2_{L^2(\partial B_\rho,a)}+\frac{1}{\rho^{n+a+2\lambda}}\lVert v_r \rVert^2_{L^2(\partial B_\rho,a)}\\&=\frac{1}{\rho^{n+a+2\lambda}}\lVert v_0-v_r \rVert^2_{L^2(\partial B_\rho,a)}+\lVert (v_r)_\rho \rVert^2_{L^2(\partial B_1,a)}\\&\le\frac{1}{\rho^{n+a+2\lambda}}\lVert v_0-v_r \rVert^2_{L^2(\partial B_\rho,a)}+C\omega(\rho).
			\eea 
			Finally, it is sufficient to choose $\rho>0$ small enough and choose a corresponding $r=r(\rho)>0$ small enough to obtain a contradiction.
		\end{proof}
	\end{proposition}
	\subsection{Regularity of blow-ups}
	Roughly speaking, we prove the regularity of the map that to any point $x\in \Gamma_\lambda(u)\cap K$ associates the homogeneous blow-up of $v=u^{x}$ at $x$, where $\lambda=1+s$ or $\lambda=2m$. We are able to prove the regularity with an explicit modulus of continuity that depend on the right-hand side of \eqref{homdecays} and \eqref{homdecaym}.
	\begin{proposition}\label{prop3} Let $u$ be a solution of \eqref{fract2} and $v=u^{x}$ be the solution of \eqref{fract5} with $x\in\Gamma_\lambda(u)$, $\vf\in C^{k,\gamma}(\R^n)$, $k\ge2$ and $\gamma\in(0,1)$. Let $v_{0,x}^{(\lambda)}$ be the $\lambda-$homogeneous blow-up of $v=u^{x}$ at $x$ and $K\subset \Gamma_\lambda(u)\cap \R^{n+1}$ be a compact set. \begin{enumerate} 
			\item If $\lambda=1+s$ and $$v_{0,x}^{(1+s)}=\lambda_x h_{e(x)}^s,$$ with $\lambda_x>0$, $e(x)\in\partial B'_1$ and $h_e^s$ as in \eqref{he}, then up to decreasing $\alpha$ in Proposition \ref{homdec}, it holds \be\label{reg2}\lvert \lambda_{x_1} -\lambda_{x_2}\rvert\le C\lvert x_1-x_2\rvert^\alpha\ee and \be\label{reg}\lvert e(x_1) -e(x_2)\rvert\le C\lvert x_1-x_2\rvert^\alpha\ee for all $x_1,x_2\in\Gamma_{1+s}(u)\cap K$.
			
			\item If $\lambda=2m\le k$ and $$v_{0,x}^{(2m)}=\lambda_xp_x $$ with $\lambda_x>0$ and $p_x$ a $2m-$homogeneous polynomial such that $\lVert p_x\rVert\ab=1$, then \be\label{reg3}\lvert \lambda_{x_1} -\lambda_{x_2}\rvert\le C(-\log(\lvert x_1-x_2\rvert)^{-\frac{1-\beta}{2\beta}})\ee and \be\label{reg1}\lVert p_{x_1} -p_{x_2}\rVert_{L^\infty( B_1)}\le C(-\log(\lvert x_1-x_2\rvert)^{-\frac{1-\beta}{2\beta}})\ee for all $x_1,x_2\in\Gamma_{2m}(u)\cap K$.
		\end{enumerate}
		\begin{proof}
			First note that the characterization of blow-up and $\lambda_x>0$ follows by Proposition \ref{cha} and Proposition \ref{nondeg}.
			
			Secondly, we can use \eqref{h'} to get \bea\frac{d}{dr}\left(\frac{H(r)}{r^{n+a+2\lambda}}\right)&=\frac{H'(r)}{r^{n+a+2\lambda}}-\frac{n+a+2\lambda}{r^{n+a+2\lambda+1}}H(r)\\&=\frac{(n+a)H(r)}{r^{n+a+2\lambda+1}}+
			2\frac{\mathcal{I}(r)}{r^{n+a+2\lambda}}-\frac{n+a+2\lambda}{r^{n+a+2\lambda+1}}H(r)\\&=2\frac{\mathcal{I}(r)}{r^{n+a+2\lambda}}-2\lambda\frac{ H(r)}{r^{n+a+2\lambda+1}}\\&=2\frac{\W^{x_0}_\lambda(r,v)}{r}.
			\eea Then, for $r$ small enough, integrating the last equality, we obtain
			$$\frac{H^{x_i}(r)}{r^{n+a+2\lambda}}-H^{x_i}_0\le C\omega_\alpha(r),$$ where we denote by $$\omega_\alpha(r)=\begin{cases} r^\alpha &\mbox{ if } \lambda=1+s\\
				(-\log(r)^{-\frac{1-\beta}{2\beta}}) &\mbox{ if } \lambda=2m.
			\end{cases}$$ the modulus of continuity. We indicate explicitly the dependence on $\alpha$ since it must be reduced to obtain the final claim. 
			
			Thus, since $H^{x_i}_0=\frac{\lambda_{x_i}^2}{c_0},$ for some dimensional constant $c_0>0$, we obtain \be\label{perdim0}c_0\frac{H^{x_i}(r)}{r^{n+a+2\lambda}}-\lambda_{x_i}^2\le C\omega_\alpha(r).\ee
			Let now $x_1,x_2\in\Gamma_\lambda(u)\cap K$ and $r=\lvert x_1-x_2\rvert^\sigma$ with $\sigma\in(0,1)$ to choose. Let $w_r$ be the homogeneous rescalings in \eqref{hom} of $w=u-\vf$. Using the regularity of the solution and $\nabla w(x_2,0)=0$, we deduce that
			\be\begin{aligned}\label{perdim-1}
				&\int_{\partial B_1} \lvert w_{r,x_1}-w_{r,x_2}\rvert\y \,d\mathcal{H}^n\\&\le\frac{1}{r^\lambda}\int_{\partial B_1} \int_0^1\lvert\nabla w((t(x_1+rx)+(1-t)(x_2+rx),ry) \rvert \lvert x_1-x_2\rvert\,dt\y\,d\mathcal{H}^n\\&\le C\frac{(\lvert x_1- x_2\lvert^{s}+r^s)\lvert x_1-x_2\rvert}{r^\lambda}\le C\frac{\lvert x_1-x_2\rvert}{r^{\lambda-s}}=C\lvert x_1-x_2\rvert^{1-\sigma(\lambda-s)}\\&=\begin{cases}
					\lvert x_1-x_2 \rvert^{1-\sigma} &\mbox{ if } \lambda=1+s\\
					\lvert x_1-x_2 \rvert^{1-2m\sigma+\sigma s} &\mbox{ if } \lambda=2m,
				\end{cases}\\&= C\omega_{1-\sigma}(\lvert x_1-x_2\rvert),
			\end{aligned}
			\ee
			where we can choose, for example, $\sigma=\frac{1}{2m}$ for $\lambda=2m$.
			
			Now, recalling $Q ^{x_i}$ as in Lemma \eqref{stima}, we have \bea\int_{\partial B_1} \lvert  Q^{x_1}_{r,x_1}&-  Q^{x_2}_{r,x_2}\rvert\y \,d\mathcal{H}^n\\&\le\int_{\partial B_1} \lvert  Q^{x_1}_{r,x_1}-  Q^{x_2}_{r,x_1}\rvert\y \,d\mathcal{H}^n+\int_{\partial B_1} \lvert   Q^{x_2}_{r,x_1}-  Q^{x_2}_{r,x_2}\rvert\y \,d\mathcal{H}^n\\&\le C\omega_{k+\gamma-\lambda}(r)+ C\omega_{1-\sigma}(\lvert x_1-x_2\rvert),
			\eea where in the last inequality we have used Lemma \ref{stima} for the first term, and the same computation in \eqref{perdim-1} for the second term.
			Therefore, recalling the definition of $u^{x_i}$ that is a solution of \eqref{fract4}, 
			we deduce that
			\be\label{perdim1}\int_{\partial B_1} \lvert u^{x_1}_{r,x_1}-u^{x_2}_{r,x_2}\rvert\y \,d\mathcal{H}^n\le C\omega_{1-\sigma}(\lvert x_1-x_2\rvert),
			\ee if $\sigma\in(0,1)$ such that $(k+\gamma-\lambda)\sigma=1-\sigma$ in the case $\lambda=1+s$.
			
			By the estimates \eqref{perdim0} and \eqref{perdim1} (with the regularity of the solution), we obtain 
			\begingroup\allowdisplaybreaks
			\begin{align*}
				\lvert\lambda_{x_1}^2&-\lambda_{x_2}^2\rvert\\&\le\left\lvert\lambda_{x_1}^2-c_0\frac{H^{x_1}(r)}{r^{n+a+2\lambda}}\right\rvert+c_0\left\lvert\frac{H^{x_1}(r)}{r^{n+a+2\lambda}}-\frac{H^{x_2}(r)}{r^{n+a+2\lambda}}\right\rvert+\left\lvert c_0\frac{H^{x_2}(r)}{r^{n+a+2\lambda}}-\lambda_{x_2}^2\right\rvert\\&\le C\omega_{\alpha\sigma}(r)+c_0\left\lvert H(1,u^{x_1}_{r,x_1})-H(1,u^{x_2}_{r,x_2})\right\rvert\\&\le C\omega_\alpha(r)+C\omega_{1-\sigma}(\lvert x_1-x_2\rvert)\\&=C\omega_{\delta}(\lvert x_1-x_2\rvert),
			\end{align*}
			\endgroup
			 if we choose $\delta=\alpha\sigma\wedge(1-\sigma)$.
			
			Since the function $x\mapsto\lambda^2_x$ is $\omega-$continuous and $K$ is a compact set, we get that \eqref{reg2} and \eqref{reg3} hold.
			
			Furthermore, we denote by $u_{0,x_i}$ the homogeneous blow-up of the rescalings in \eqref{hom} of the function $u^{x_i}$. By \eqref{homdecays}, \eqref{homdecaym} and \eqref{perdim1}, we get
			\begin{equation} \label{perdim3}
				\begin{aligned}
					\int_{\partial B_1} \lvert u_{0,x_1}-u_{0,x_2}\rvert& \y\,d\mathcal{H}^n\\&\le\int_{\partial B_1} \lvert u_{0,x_1}-u^{x_1}_{r,x_1}\rvert\y \,d\mathcal{H}^n+\int_{\partial B_1} \lvert u^{x_1}_{r,x_1}-u^{x_2}_{r,x_2}\rvert \y\,d\mathcal{H}^n\\&\qquad+\int_{\partial B_1} \lvert u^{x_2}_{r,x_2}-u_{0,x_2}\rvert \y\,d\mathcal{H}^n\\&\le \omega_\alpha(r)+\omega_{1-\sigma}(\lvert x_1-x_2\rvert)=\omega_{\delta}(\lvert x_1-x_2\rvert),
				\end{aligned} 
			\end{equation}since $\delta=\alpha\sigma\wedge(1-\sigma)$.
			
			Notice that the blow-ups $u_{0,x_1}$ and $u_{0,x_2}$ are solutions of \eqref{fract4} (with 0 obstacle). Therefore $$\int_{B_1} L_a (u_{0,x_i})u_{0,x_i}\y\,dX=0$$ and $$\int_{B_1} L_a (u_{0,x_i})u_{0,x_j}\y\,dX=0$$ for all $i,j\in\{1,2\},$ and so $$
			\int_{B_1} L_a (u_{0,x_1}-u_{0,x_2})(u_{0,x_1}-u_{0,x_2})\y\,dX\ge0. $$ 
			Now, integrating by parts, we obtain 
			\be\label{perdim4}\begin{aligned} \int_{B_1}\lvert \nabla(u_{0,x_1}-u_{0,x_2}) \rvert^2 \y\,dX&\le\int_{\partial B_1}\partial_\nu (u_{0,x_1}-u_{0,x_2})(u_{0,x_1}-u_{0,x_2})\y\,d\HH^n\\&=\lambda\int_{\partial B_1}\lvert u_{0,x_1}-u_{0,x_2}\rvert^2\y\,d\HH^n\\&\le\lambda\int_{ B_1}\lvert u_{0,x_1}-u_{0,x_2}\rvert^2\y\,dX
			\end{aligned}
			\ee
			where we used that $u_{0,x_1}-u_{0,x_2}$ is $\lambda-$homogeneous.
			
			Hence, by the homogeneity of $u_{0,x_1}-u_{0,x_2}$ and by Proposition \ref{embedding}, it follows that \be\label{perdim6}\begin{aligned}\int_{\partial B'_1} \lvert u_{0,x_1}-u_{0,x_2}\rvert^2\,d\HH^{n-1}&\le C\int_{B'_1} \lvert u_{0,x_1}-u_{0,x_2}\rvert^2\,d\HH^n\\&\le C\int_{B_1}\lvert u_{0,x_1}-u_{0,x_2}\rvert^2\y\,dX\\&\qquad+C\int_{B_1}\lvert\nabla( u_{0,x_1}-u_{0,x_2})\rvert^2\y\,dX\\&\le \omega_{\delta}(\lvert x_1-x_2\rvert)
			\end{aligned}
			\ee where in the last inequality we used \eqref{perdim3} (with the regularity of the solution) and \eqref{perdim4}.

			Using \eqref{reg2}, \eqref{reg3} and \eqref{perdim6}, we get
			\begin{equation}\label{perdim7}
				\begin{aligned}
					\int_{\partial B'_1} \left\lvert \frac{u_{0,x_1}}{\lambda_{x_1}}-\frac{u_{0,x_2}}{\lambda_{x_2}}\right\rvert\,d\mathcal{H}^{n-1}&\le\int_{\partial B'_1} \left\lvert \frac{u_{0,x_1}}{\lambda_{x_1}}-\frac{u_{0,x_2}}{\lambda_{x_1}}\right\rvert \,d\mathcal{H}^{n-1}\\&\qquad+\int_{\partial B'_1} \left\lvert \frac{u_{0,x_2}}{\lambda_{x_1}}-\frac{u_{0,x_2}}{\lambda_{x_2}}\right\rvert \,d\mathcal{H}^{n-1}\\&\le C\omega_{\delta}^\frac12(\lvert x_1-x_2\rvert)\\&\qquad+\left\lvert\frac{\lambda_{x_1}-\lambda_{x_2}}{\lambda_{x_1}}\right\rvert
					\int_{\partial B'_1}\frac{ \left\lvert u_{0,x_2}\right\rvert}{\lvert \lambda_{x_2}\rvert} \,d\mathcal{H}^{n-1}\\&\le C\omega_{\delta}^\frac12(\lvert x_1-x_2\rvert),
				\end{aligned}
			\end{equation}
			where we used the $\omega-$continuity of the function $x\mapsto \lambda_x$ for $x\in K$ to estimate from below $\lambda_{x_1}$.
			
			Now, by \eqref{perdim7} and the definition of $h_e^s$ in \eqref{he}, for $\lambda=1+s$ we have $$\int_{\partial B'_1}\lvert (x\cdot e_{x_1})_+-(x\cdot e_{x_2})_+\rvert\,d\mathcal{H}^{n-1}\le Cr^{\frac\delta2},$$ which implies \eqref{reg}.
			
			Finally, for $\lambda=2m$, since all norms in a finite dimensional space are equivalent, we get $$\lVert p_{x_1} -p_{x_2}\rVert_{L^\infty( B_1)}\le C \int_{\partial B_1}\lvert p_{x_1}-p_{x_2}\rvert\y\,d\mathcal{H}^{n}.$$ By a similar computation as in \eqref{perdim7}, we obtain $$ \int_{\partial B_1}\lvert p_{x_1}-p_{x_2}\rvert\y\,d\mathcal{H}^{n}\le C (-\log(\lvert x_1-x_2\rvert)^{-\frac{1-\beta}{2\beta}}),$$ where we used \eqref{perdim3}, then we conclude \eqref{reg1}
		\end{proof}
	\end{proposition}
	
	\subsection{Proof of Theorem \ref{thm3}}
	For the regularity of $\Gamma_{1+s}(u)$ and the structure of $\Gamma_{2m}(u)$ for $2m\le k$ we can proceed with a standard argument, as in \cite{fs16}, \cite{gps16}, \cite{csv17}, \cite{gpps17}, \cite{gp09}. We briefly recall the proof.
	
	\textbf{Step 1.} Let's start with part (1), i.e. $\Gamma_{1+s}(u)$ is locally a $C^{1,\alpha}$ submanifold of dimension $n-1$. Without loss of generality, we can assume $0\in\Gamma_{1+s}(u)$ and prove the $C^{1,\alpha}$ regularity in a neighborhood of 0.
	
	Since $\Gamma_{1+s}(u)$ is relatively open in $\Gamma(u)$\footnote{It is sufficient to use the frequency gap in Proposition \ref{gap} and the upper semicontinuity of the function $x_0\mapsto\Phi^{x_0}(0^+,u)$, since is an infimum of continuous functions.}, there is $ \rho>0$ such that $\Gamma_{1+s}(u)\cap B'_\rho=\Gamma(u)\cap B'_\rho$, and, up to decreasing $\rho$ and up to a rotation, we can assume \be\label{per0}e(0)=e_n\mbox { and }e(x)\cdot e_n\ge\frac12\ee for all $x\in\Gamma_{1+s}(u)\cap B'_\rho$.
	
	We define the following cones centered in $x_0\in\Gamma_{1+s}(u)\cap B'_\rho$
	$$C^\pm(x_0,\eps)=\{(x,0)\in\R^{n}\times\{0\}:(x-x_0)\cdot e(x_0)\ge \eps\lvert x-x_0\rvert\}.$$
	There is $\eps_0>0$ such that for all $\eps\in(0,\eps_0)$ there is $\delta_\eps>0$ such that \be\label{per1}C^+(x_0,\eps)\cap B'_{\delta_{\eps}}\subset \Lambda^C(u)=\{u(x,0)>0\}\ee and \be\label{per2}C^-(x_0,\eps)\cap B'_{\delta_{\eps}}\subset \Lambda(u)=\{u(x,0)=0\}.\ee
	
	Now let $\varepsilon_0 >0$, $\varepsilon \in(0,\varepsilon _0)$ and $\delta_\varepsilon >0$ as above, therefore we define $$g(x'')=\max\{t\in\mathbb{R}:(x'',t,0)\in\Lambda(u))\}$$ with $x''\in\mathbb{R}^{n-1}$ and $\lvert x''\rvert\le\delta_\varepsilon \sqrt{(1-\varepsilon ^2)}=:r_\eps$. 
	
	Furthermore $$\Gamma_{1+s}(u)\cap B'_\rho=\{(x,0)\in\R^{n}\times\{0\}:x_n=g(x_1,\ldots,x_{n-1})\}\cap B'_\rho.$$	
	Thus, using \eqref{per1} and \eqref{per2} we get that $g$ is Lipschitz in $B''_{r_\eps}\subset\R^{n-1}$ with $$\nabla g(x'')=-\frac{e''(x)}{e(x)\cdot e_n},$$ and finally that $g\in C^{1,\alpha}(B''_{r_\eps})$ by \eqref{reg} and \eqref{per0}.

	\textbf{Step 2.} Now we prove part (2), i.e. the structure of $\Gamma_{2m}(u)$ for $2m\le k$, in particular we prove that $\Gamma_{2m}^j(u)$, defined in Theorem \ref{thm3}, is locally contained in a $C^{1,\log}$ submanifold of dimension $j$, with $$\Gamma_{2m}^j(u):=\{x_0\in \Gamma_{2m}:d_{2m}^{x_0}=j\},$$ where $d_{2m}^{x_0}$ is as in \eqref{d}, according to Proposition \ref{nondeg}. 
	
	Let $v_{0,x_0}=\lambda_{x_0}p_{x_0}$ be the only homogeneous blow-up, as in Proposition \ref{prop3}, with $x_0\in \Gamma_{2m}^j(u)\cap K$ where $K\subset\Gamma_{2m}^j(u)\cap \R^{n+1}$ is a compact set. The function $$q_{x_0}(x,y)=p_{x_0}(x-x_0,y)$$ satisfies the hypotheses of a modified version of the Whitney extension theorem (see the original in \cite{whitney}) for a $C^{m,\omega}$ function (see \cite{fef}), where $\omega$ in our case is a logarithmic modulus of continuity.
	
	Thus, using \eqref{reg1}, we can extend $q_{x_0}$ to a function $F\in C^{2m,\log}(\R^{n+1})$ such that $$D^\alpha F(x_0,0)=D^\alpha q_{x_0}(x_0,0)=D^\alpha p_{x_0}(0).$$
	
	Since the blow-up is non-degenerate, by Proposition \ref{nondeg}, there exist $f_i\in\mathbb{R}^n$, $i=1,\ldots,n-j$, linearly independent, such that  $$f_i\cdot \nabla_x p_{x_0}(x,y)\not=0,$$ for some $(x,y)\in\mathbb{R}^{n+1}$, by definition of $\Gamma_{2m}^j(u)$.
	
	Then $\exists \beta_i\in\mathbb{N}^n$ for $i=1,\ldots,n-j$ such that $\lvert \beta_i\rvert=2m-1$ and \be\label{per6}\partial_{f_i}D^{\beta_i}F(x_0,0)=\partial_{f_i}D^{\beta_i} p_{x_0}(0)\not=0,\ee since $\partial_{f_i} p_{x_0}$ is a $L_a-$harmonic polynomial of degree $2m-1$, it cannot vanish on $\{y=0\}$ by Lemma \ref{ex}. Then it is sufficient to choose the multiindex $\beta_i$ such that $\partial _{f_i} p_{x_0}$ contains the monomial $x^{\beta_i}$.
	
	Moreover $D^{\beta_i}F(x_0,0)=D^{\beta_i} p_{x_0}(0)=0,$ since $ p_{x_0}$ is a $2m-$homogeneous polynomial. Thus, we deduce that $$\Gamma_{2m}^j(u)\cap K\subset\bigcap_{i=1}^{n-j}\{D ^{\beta_i}F(\cdot,0)=0\},$$ since $x_0\in \Gamma_{2m}^j(u)\cap K$ is arbitrary. 
	
	We can use the implicit function theorem in any neighborhood of $x_0\in \Gamma_{2m}^j(u)\cap K$, since \eqref{per6} holds. Hence $\exists U\subset \mathbb{R}^j$ and $B_r(x_0)\subset\mathbb{R}^n$ with $r=r(x_0)>0$ such that $\bigcap_{i=1}^{n-j}\{D ^{\beta_i}F(\cdot,0)=0\}$ is locally a graph of a function $\varphi: U\subset \mathbb{R}^j\to\mathbb{R}^{n-j}$.
	
	Finally, since $D ^{\beta_i}F\in C^{1,\log}(\mathbb{R}^{n+1})$, we obtain that $\varphi$ is $C^{1,\log}(U)$. As a consequence, $\Gamma_{2m}^j(u)\cap B_r(x_0)$ is contained in a $j-$dimensional submanifold $C^{1,\log}$, which concludes the proof.

	\appendix
	\section{Appendices}

	In the following proposition, we see that a solution $u$ of \eqref{fract4} (with 0 obstacle) is a solutions of $L_a u=f$ in the whole ball $B_1$ with right-hand side $f$ which is a measure depending on $u$.
	\begin{proposition}\label{prop1} Let $u\in H^1(B_1,a)$ be a solution of $L_au=0$ in $B_1^+$ and even in the $y$ direction. %in a weak sense in the space $H^1(B_1,a)$.ù
		Then \be \label{measure} \mbox{div}(\y\nabla u)=2\lim_{y\to0^+}(\y \partial_y u)\HH^n|_{\{y=0\}}.\ee %Moreover, the following integration by parts formula holds: \be\label{div}\int_{B_1}\nabla u\cdot \nabla \eta \y\,dX=\int_{\partial B_1}\eta\partial_\nu u\y\,d\HH^n-2\int_{B_1'}\eta\lim_{y\to0^+}(\y \partial_y u)\,d\HH^n.\ee
		\begin{proof} This is a simple consequence of an integration by parts, but for the sake of completeness we give the complete proof. 
			
			We compute $\mbox{div}(\y\nabla u)$ in distributional sense, as following \bea<-\mbox{div}(\y\nabla u),\eta> &=2\int_{B_1^+} \nabla u \cdot \nabla \eta\y\,dX\\&=2 \lim_{\varepsilon\to0^+}\int_{B_1^+\cap\{y\ge\varepsilon\}} \nabla u \cdot \nabla \eta\y\,dX\\&=2\lim_{\varepsilon\to0^+}\int_{B^+_1\cap\{y=\varepsilon\}}\eta \partial_\nu u\eps^a\,d\mathcal{H}^n\\&=-2\int_{B'_1}\eta\lim_{\eps\to0^+}(\partial_y u(x',\eps)\eps^a)\,d\mathcal{H}^n,\eea where we used the integration by parts \bea\int_{B_1}\nabla u\cdot \nabla \eta \y\,dX&=\int_{\partial B_1}\eta \partial_\nu u\y\HH^n-\int_{B_1}\mbox{div}(\y\nabla u)\eta\,dX\eea
			and $-\mbox{div}(\y\nabla u)=0$ in $B_1^+\cap\{y\ge\varepsilon\}$.
			
			%Again by parts and by \eqref{measure} we obtain \eqref{div}.
		\end{proof}
	\end{proposition}

	In the following lemma we prove some useful identities.
\begin{lemma}
	Let $w\in H^1(B_1,a)$ be $\lambda-$homogeneous. Then 
	
	\begin{equation}\label{w0} \int_{B_1} w\y\,dX=\frac1{n+a+\lambda+1}\int_{\partial B_1} w\y\,d\mathcal{H}^n,	
	\end{equation}
	\begin{equation}\label{w1} \lVert \nabla w\rVert^2 \aB=\frac{1}{n+a+2\lambda-1}\lVert \nabla_\theta w\rVert^2 \ab+\frac{\lambda^2}{n+a+2\lambda-1}\lVert w\rVert^2 \ab
	\end{equation}
	
	and \begin{equation}\label{w2}
		W_\lambda(w)=\frac{1}{n+a+2\lambda-1}\left(\lVert \nabla_\theta w\rVert^2 \ab-\lambda(\lambda+n+a-1)\lVert w\rVert^2 \ab\right),
	\end{equation}
	where $\nabla_\theta w(x)$ is the gradient of $w$ on $\partial B_1$.
	
	Moreover, if $w$ is also a solution of \eqref{fract4} (with 0 obstacle), then 
	\begin{equation}\label{w3}
		\lVert \nabla_\theta w\rVert^2 \ab=\lambda(\lambda+n+a-1)\lVert w\rVert^2 \ab,
	\end{equation}
	in particular
	\begin{equation}\label{w4}
		W_\lambda(w)=0.
	\end{equation}
	\begin{proof}
		The proof is a straightforward computation.  
	\end{proof}
	
\end{lemma}

The following lemma allows us to reduce the problem \eqref{fract2} for obstacle $\vf\not\equiv0$, to the problem \eqref{fract5} with 0 obstacle and right-hand side $h$.

\begin{lemma}\label{ex} Let $q_k(x)$ be an homogeneous polynomial of degree $k$. Then there is a unique $\widetilde q_k(x,y)$ homogeneous polynomial of degree $k$ such that
	\be
	\begin{cases}\label{y}
		L_a \widetilde  q_k(x,y)=0& \mbox{in } \R^{n+1}\\
		\widetilde q_k(x,y)=\widetilde  q_k(x,-y)& \mbox{in } \R^{n+1}\\
		\widetilde q_k(x,0)=q_k(x) & \mbox{in } \R^n\\
	\end{cases}
	\ee
	In particular, an homogeneous polynomial of degree $k$, that satisfies \eqref{y} and vanishes identically on $\{y=0\}$, must vanish identically on $\R^{n+1}$.
	\begin{proof}
		See Lemma 5.2 in \cite{gr17}.
	\end{proof}
\end{lemma}
	
	The following theorem is a generalization of Poincaré Theorem in weighted Sobolev space $H^1(B_R,a).$
	\begin{theorem} Let $w\in H^1(B_R,a)$. Then \be\label{poinc}\int_{B_R}w^2\y\,dX\le C\left(\int_{B_R}\lvert \nabla w\rvert^2\y\,dX+\int_{\partial B_R}w^2\y\,d\HH^n\right),\ee with $C=C(n,a)>0.$
		\begin{proof}
			See Lemma 2.10 in \cite{css08}.
		\end{proof}
	\end{theorem}

	We recall the following generalization of the Liouville's Theorem for entire $L_a-$harmonic functions.
	\begin{theorem}[Liouville Theorem]\label{liouville} Let $w$ be a global solution of $L_aw(x,y)=0$ for $(x,y)\in \R^n\times\R$ such that $w$ is even in the $y$ direction and $$\lvert w(x,y)\rvert\le C \lvert (x,y)\rvert ^\alpha.$$ 
		Then, $w$ is a polynomial.
		\begin{proof}
			See Lemma 2.7. in \cite{css08}.
		\end{proof}
	\end{theorem}
	
	In the following Proposition, we show an embedding from $H^1(B_1,a)$, the weighted Sobolev space in $B_1$, to $L^2(B'_1)$, the Lebesgue space on $B'_1$.
	\begin{proposition}\label{embedding} If $a\in(-1,1)$, then there is a bounded operator $T:H^1(B_1,a)\to L^2(B'_1) $, i.e. $$\lVert w \rVert _{L^2(B'_1)}\le C\lVert w \rVert _{H^1(B_1,a)}, $$ for all $w\in H^1(B_1,a)$.
		
		Moreover \bea\lVert \phi \rVert _{L^2(\partial B'_1)}\le C\lVert \phi \rVert _{H^1(\partial B_1,a)}\eea for all $\phi\in H^1(\partial B_1,a)$.
		\begin{proof}
			See Theorem 2.8. in \cite{nek}.
		\end{proof}
	\end{proposition}
	\begin{lemma}\label{stima}
		Let $\vf\in C^{k,\gamma}(\R^n)$ and $x_1,x_2\in\R^n$. Let $Q ^{x_i}(x,y)=\widetilde q_k^{x_i}(x,y)-q_k^{x_i}(x)$, where $q_k^{x_i}$ is the $k$-th Taylor polynomial of $\vf$ at $x_i$ and $\widetilde q_k^{x_i}$ is the extension according to Lemma \ref{ex}.
		If $r=\lvert x_1-x_2\rvert^\sigma$ for some $\sigma\in(0,1)$, then 
		\be
		\int_{\partial B_1} \lvert Q^{x_1}_{r,x_1}-Q^{x_2}_{r,x_1}\rvert\y \,d\mathcal{H}^n\le Cr^{k+\gamma-\lambda},
		\ee  where $Q^{x_i}_{r,x_1}$ are the rescalings of $Q^{x_i}$ as in \eqref{hom}.
		\begin{proof}
				We denote by $q$ the function $q_k$ and we denote by $\widetilde q$ the function $\widetilde q_k$. Moreover we consider $q^{x_i}_{r,x_i}$ the rescalings of $q^{x_i}$ as in \eqref{hom}.
				
				Notice that, by $C^{k,\gamma}-$regularity of $\vf,$ we get $$\lVert \vf-q^{x_1}\rVert_{L^\infty(B_r(x_1))}\le C{r}^{k+\gamma}$$ and $$\lVert \vf-q^{x_2}\rVert_{L^\infty(B_r(x_1))}\le \lVert \vf-q^{x_2}\rVert_{L^\infty(B_{2r}(x_2))}\le C{r}^{k+\gamma},$$ then
				\be\label{k}\lVert q^{x_1}-q^{x_2}\rVert_{L^\infty(B_r(x_1))}\le Cr^{k+\gamma}.\ee
				We deduce that
				\bea
				\int_{\partial B_1} \lvert q^{x_1}_{r,x_1}-q^{x_2}_{r,x_1}\rvert&\y \,d\mathcal{H}^n\\&=\frac{1}{r^\lambda}\int_{\partial B_1} \lvert q^{x_1}(x_1+rx,ry)-q^{x_2}(x_1+rx,ry)\rvert \y\,d\mathcal{H}^n\\&=\frac{1}{r^{\lambda+n+a}}\int_{\partial B_r} \lvert q^{x_1}(z)-q^{x_2}(z)\rvert \lvert z_{n+1}\rvert^a\,d\mathcal{H}^n
				\\&\le \frac{C}{r^\lambda}\lVert q^{x_1}-q^{x_2}\rVert_{L^\infty(\partial B_r(x_1))}\le Cr^{k+\gamma-\lambda},
				\eea where we used \eqref{k} in the last inequality.
			
				Notice that the space of $L_a-$harmonic polynomials in $\R^{n+1}$ of degree $k$ which are even in the $y$ direction is a finite dimensional space. Then $$\lVert \widetilde p\rVert_{L^1(B_1,a)}\le C\lVert p\rVert_{L^1(B'_1,a)}$$ for each $p$ which is a polynomial of degree $k$, $L_a-$harmonic and even in the $y$ direction. Notice that the right-hand side is a norm by Lemma \ref{ex}.
			
				Consider the operation of extension as in Lemma \ref{ex} and the operation of rescaling as in \eqref{hom}. These operations commute, by the explicit formulation of the extension, defined in Lemma 5.2 in \cite{gr17}.
				
				Therefore, we have
				\bea \int_{\partial B_1} \lvert \widetilde q^{x_1}_{r,x_1}-\widetilde q^{x_2}_{r,x_1}\rvert\y \,d\mathcal{H}^n&\le \int_{B_1} \lvert \widetilde q^{x_1}_{r,x_1}-\widetilde q^{x_2}_{r,x_1}\rvert\y \,dX\\&\le C\int_{B'_1} \lvert q^{x_1}_{r,x_1}- q^{x_2}_{r,x_1}\rvert\y \,d\HH^{n}\\&\le C\int_{B_1} \lvert q^{x_1}_{r,x_1}- q^{x_2}_{r,x_1}\rvert\y \,dX\\&\le\frac{C}{r^\lambda}\lVert q^{x_1}-q^{x_2}\rVert_{L^\infty(B_r(x_1))}\le C r^{k+\gamma-\lambda},
				\eea by \eqref{k}, which concludes the proof.
		\end{proof}
	\end{lemma}
	
	\bibliographystyle{alpha}
	\bibliography{paper1.bib}
\end{document}